\documentclass[11pt,twoside]{article}
\usepackage[latin9]{inputenc}
\usepackage[a4paper]{geometry}
\usepackage{array}
\usepackage{float}
\usepackage{multirow}
\usepackage{amsmath}
\usepackage{amsthm}
\usepackage{amssymb}
\usepackage{graphicx}
\usepackage{esint}
\usepackage{mathtools}
\usepackage{commath}
\usepackage[sc,osf]{mathpazo}
\usepackage{hyperref} 
\makeatletter
\providecommand{\tabularnewline}{\\}
\@ifundefined{showcaptionsetup}{}{%
 \PassOptionsToPackage{caption=false}{subfig}}
\usepackage{subfig}
\begin{document}
\title{Least-squares formulations for Stokes equations with non-standard boundary conditions - A unified approach}
\author{S. Mohapatra\thanks{IIIT Delhi, Email: subhashree@iiitd.ac.in}$^{\star}$,
N. Kishore Kumar\thanks{BITS-Pilani Hyderabad Campus, Hyderabad, Email: naraparaju@hyderabad.bits-pilani.ac.in}$^{\dagger}$,
Shivangi Joshi\thanks{BITS-Pilani Hyderabad Campus, Hyderabad, Email: p20200036@hyderabad.bits-pilani.ac.in}$^{\dagger}$
}
\date{~}
\maketitle
\def \R{{{\rm I{\!}\rm R}}}
\newtheorem{thm}{Theorem}[section]
\newtheorem{prop}{Proposition}[section]
\newtheorem{lem}{Lemma}[section]
\newtheorem{coro}{Corollary}[section]
\newtheorem{rem}{Remark}[section]
\newtheorem{guess8}{Example}[section]
\newcommand{\vb}{\mbox{\boldmath$v$}}
\newcommand{\eb}{\mbox{\boldmath$e$}}
\newcommand{\ub}{\mbox{\boldmath$u$}}
\newcommand{\Ub}{\mbox{\boldmath$U$}}
\newcommand{\Vb}{\mbox{\boldmath$V$}}
\newcommand{\Fb}{\mbox{\boldmath$F$}}
\newcommand{\Zb}{\mbox{\boldmath$Z$}}
\newcommand{\Ab}{\mbox{\boldmath$A$}}
\newcommand{\Gb}{\mbox{\boldmath$G$}}
\newcommand{\fb}{\mbox{\boldmath$f$}}
\newcommand{\tb}{\mbox{\boldmath$t$}}
\newcommand{\zb}{\mbox{\boldmath$z$}}
\newcommand{\lb}{\mbox{\boldmath$l$}}
\newcommand{\gb}{\mbox{\boldmath$g$}}
\newcommand{\hb}{\mbox{\boldmath$h$}}
\newcommand{\ob}{\mbox{\boldmath$o$}}
\newcommand{\Hb}{\mbox{\boldmath$H$}}
\newcommand{\Rb}{\mbox{\boldmath$R$}}
\newcommand{\ab}{\mbox{\boldmath$a$}}
\newcommand{\xb}{\mbox{\boldmath$x$}}
\newcommand{\kb}{\mbox{\boldmath$k$}}
\newcommand{\sbb}{\mbox{\boldmath$s$}}
\newcommand{\phb}{\mbox{\boldmath$\alpha$}}
\newcommand{\bb}{\mbox{\boldmath$b$}}
\newcommand{\betab}{\mbox{\boldmath$\beta$}}
\newcommand{\omegab}{\mbox{\boldmath$\omega$}}
\newcommand{\sigmab}{\mbox{\boldmath$\sigma$}}
\newcommand{\varb}{\mbox{\boldmath$\varrho$}}
\newcommand{\BBb}{\mbox{\boldmath$\mathfrak{B}$}}
\newcommand{\nuw}{\mbox{\boldmath$\nu$}}
\newcommand{\taub}{\mbox{\boldmath$\tau$}}
\newcommand{\etav}{\mbox{\boldmath$\eta$}}
\newcommand{\zetau}{\mbox{\boldmath$\zeta$}}
\newcommand{\Pib}{\mbox{\boldmath$\Pi$}}
\newcommand{\Phiy}{\mbox{\boldmath$\Phi$}}
\newcommand{\psy}{\mbox{\boldmath$\psi$}}
\newcommand{\Tht}{\mbox{\boldmath$\Theta$}}
\newcommand{\wa}{\mbox{\boldmath$w$}}
\newcommand{\nb}{\mbox{\boldmath$n$}}
\newcommand{\mc}{\mbox{\boldmath$m$}}
\newcommand{\hd}{\mbox{\boldmath$h$}}
\newcommand{\fe}{\mbox{\boldmath$f$}}
\newcommand{\Bb}{\mbox{\boldmath$B$}}
\newcommand{\Cb}{\mbox{\boldmath$C$}}
\newcommand{\Db}{\mbox{\boldmath$D$}}
\newcommand{\Eb}{\mbox{\boldmath$E$}}
\newcommand{\Lb}{\mbox{\boldmath$L$}}
\newcommand{\Wb}{\mbox{\boldmath$W$}}
\newcommand{\MIG}[1]{\textcolor{red}{#1}}
\abstract{
In this paper, we propose a unified  non-conforming least-squares spectral element approach for solving Stokes equations with various 
non-standard boundary conditions. Existing least-squares formulations mostly deal with Dirichlet
boundary conditions and are formulated using ADN theory based regularity estimates. However, changing
boundary conditions lead to a search for parameters satisfying supplementing and complimenting conditions \cite{ADN} which is not easy always.  Here we have avoided ADN  theory based regularity estimates
and proposed a unified approach for dealing with various boundary conditions. Stability estimates and error estimates have been discussed. Numerical results
displaying exponential accuracy have been presented for both two and three dimensional
cases with various boundary conditions.
}
\section{Introduction}
Efficient numerical approximation of Stokes equations has been of great theoretical and
computational interest for a long time as it is the first step to investigate nonlinear
Navier Stokes equations. A search on numerical solutions of Stokes equations leads to an enormous
amount of contributions. However, very few are available for Stokes equations with non-standard boundary conditions, i.e. boundary conditions other than Dirichlet. There are several practical instances like blood flow in arteries, water flow in pipes with bifurcation, 
oil ducts etc., which need to be analyzed  other than velocity based Dirichlet boundary condition. A major challenge with changing the boundary conditions is to incorporate them in variational formulation and corresponding functional spaces.

Here we refer boundary conditions other than Dirichlet as non-standard boundary conditions
and have listed eighteen different boundary conditions, being attached with Stokes differential operator from our literature survey.
One common mathematical difficulty can be described as difficulty in understanding wellposedness of both continuous problem and discrete problem. Functional spaces need to be modified accordingly. 
Major difficulty with mixed finite element formulation with change in boundary conditions is to have appropriate saddle point formulation. 

In this paper, we propose a unified approach for solving Stokes equations with various non-standard boundary
conditions. The numerical scheme is based on a non-conforming least squares spectral element approach. Because of the least-squares 
formulation, the obtained linear system is always symmetric and positive definite in each case of boundary condition, hence we solve the linear system
using the preconditioned conjugate gradient method. Minimization of least-squares functional consists of two parts. The first part consists of residues
in differential operators and jumps in velocity and pressure components (along interelement boundaries). The second part consists of residues from the physical boundary conditions. This part varies with change in the boundary conditions. Norms are suitably chosen from
available regularity estimates.
Using the unified approach, any of these boundary conditions can be used
with the Stokes problem and just boundary related terms need to be changed in the least-squares functional. No need to change in any other functional settings. 

In section 2, we introduce Stokes equations with various non-standard boundary conditions, their physical motivations, theoretical, computational developments and available regularity estimates. 
 Stability estimates for numerical solutions have been discussed in section 3. Numerical schemes have been discussed in section 4. Error estimates are discussed in section 5. Numerical results with some of these boundary
conditions have been presented in section 6. Finally, we conclude with section 7. 

\section{Notations}
Let $\Omega\subset\mathbb{R}^{n}, n=2,3$ with sufficiently smooth boundary $\Gamma=\partial\Omega$.
Here we consider the Stokes system
\begin{align}
\begin{cases}
-\Delta\ub+\nabla p=\fb \quad\mbox{in}\quad\Omega,\\
-\nabla\cdot\ub =\chi \quad\mbox{in}\quad\Omega\\
\end{cases}\label{eq1}
\end{align}
with different boundary conditions other than just Dirichlet boundary conditions.
 Basic least-squares formulation leads to minimize the residuals in differential operators and residuals in 
 boundary conditions in appropriate norms (\cite{PR1}, \cite{PR2}, \cite{PR3}, \cite{KIM1}, \cite{DG}, \cite{BD}, 
 \cite{BD1},\cite{JN1},\cite{SM1},\cite{SM2}, \cite{CH1},\cite{CH2}, \cite{JI1}).

 Let's denote $\mathcal{L}(\ub,p)=-\Delta\ub+\nabla p, \mathcal{D}\ub=-\nabla\cdot\ub$. Let $\nb=(n_{1},n_{2}), (n_1,n_2,n_3)$ be the unit outward normal to $\Gamma=\partial\Omega$ 
for $\Omega\subset {\mathbb{R}}^{2}, {\mathbb{R}}^{3}$ respectively.
$u_{\nb}=\ub.\nb,\ub_{\tau}=\ub-\nb u_{\nb}$ denote the normal and tangential components
of the velocity. $\sigma_{\nb}=\sigma_{\nb}(\ub,p)=\sigma(\ub,p).\nb,\sigma_{\taub}=\sigma_{\taub}(\ub,p)=\sigma(\ub,p)-\nb\sigma_{\nb}(\ub,p)$
denote normal and tangential components of stress vector. Here $\sigma(\ub,p)=(-p\delta_{ij}+e_{ij}\ub)\nb, e_{ij}(\ub)= \frac{\partial u_{i}}{\partial x_{j}}+\frac{\partial u_{j}}{\partial x_{i}}$.
 $\mathbb{D}(\ub)=\frac{1}{2}(\nabla\ub+\nabla\ub^{T})$ denotes the strain tensor. 
 Let's denote $L^{2}(\Omega)/{\mathbb{R}}=\left\{q\in L^{2}(\Omega):\int_{\Omega}qd\Omega=0\right\}.$
 $c$ is treated as a generic constant, which has different values with different situations.

$H^{m}(\Omega)$ denotes the Sobolev space of functions with
square integrable derivatives of integer order less than or equal to $m$ on $\Omega$ equipped with the norm
\begin{align}
{\|u\|}_{H^{m}(\Omega)}^{2}=\sum_{\left|\alpha\right|\leq m} {\|D^{\alpha}u\|}_{{L^{2}(\Omega)}^{2}}.\notag
\end{align}
Further, let $E=(-1,1)^{d-1}$. Then we define fractional norms $(0<\mu<1)$ by :
\begin{itemize}
 \item {if $\Gamma_l\in\partial\Omega\!\!\!\quad\mbox{when}\quad\!\!\!\Omega\subset \mathbb{R}^{2}$}
\begin{align}
\left\Vert u\right\Vert _{\mu,\Gamma_{l}}^{2}=\|u\|_{0,E}^2+\int_{E}\int_{E}\frac{\left|u(\xi)-u(\xi^{\prime})\right|^{2}}
{\left|\xi-\xi^{\prime}\right|^{1+2\mu}}d\xi d\xi^{\prime}.\notag
\end{align}
\end{itemize}
\begin{itemize}
\item{if $\Gamma_l\in\partial\Omega\!\!\!\quad\mbox{when}\quad\!\!\!\Omega\subset \mathbb{R}^{3}$}
\begin{align}
\|u\|_{\mu, \Gamma_l}^2 = \|u\|_{0,E}^2 + \int_{E}\int_{E}\int_{E}
\frac{(u(\xi,\eta)-u(\xi^{\prime},\eta))^2}{(\xi-\xi^{\prime})^{1+2\mu}}\,d\xi d\xi^{\prime}d\eta\notag\\
+\int_{E}\int_{E}\int_{E}\frac{(u(\xi,\eta)-u(\xi,\eta^{\prime}))^2}
{(\eta-\eta^{\prime})^{1+2\mu}}\,d\eta d\eta^{\prime}d\xi\,.\notag
\end{align}
\end{itemize}
We  denote the vectors by bold letters. For example if $\Omega\subset \mathbb{R}^{2}$, $\ub =(u_{1},u_{2})^{T}$, $\Hb^{k}(\Omega)=H^{k}(\Omega)
\times H^{k}(\Omega)$, $\|\ub\|_{k,\Omega}^{2}=\|u_{1}\|_{k,\Omega}^{2}+\|u_2\|_{k,\Omega}^{2}$ etc.. We define the functional spaces and norms similarly
for three dimensional case.

Next, we present different boundary conditions existing in the literature that have been
encountered with Stokes equations in various physical applications. We plan to discuss
exponentially accurate least-squares spectral element formulations for Stokes equations with each of these boundary conditions. 
\begin{itemize}
 \item (B1)  $\Gamma=\bar{\Gamma}_{1}\cup\bar{\Gamma}_{2}$
               \begin{align}
               \begin{cases}
                \ub=\gb \quad\!\!\!\mbox{on}\quad\!\!\!\Gamma_{1}\\
                \ub.\nb=\gb_{\nb}, (\nabla\times\ub)\times\nb
                =\kb\times\nb\quad\!\!\!\mbox{on}\quad\!\!\!\Gamma_{2}
               \end{cases}(\cite{AM7})\notag
              \end{align}
 \item (B2) $\ub_{\nb}=\gb_{\nb},[\frac{\partial\ub}{\partial\nb}-p\nb+b\ub]_{\taub}=\hb_{\taub}$ on $\Gamma$ (\cite{ME1},\cite{ME2})
 \item (B3) $\ub_{\nb}=\gb_{\nb},\big([\nabla\ub+(\nabla\ub)^{T}-pI]\nb+b\ub\big)_{\taub}=\hb_{\taub}$ on $\Gamma$ (\cite{ME1},\cite{ME2})
 \item (B4) $\ub.\nb=g$, $(\nabla\times\ub)\times\nb=\hb\times\nb$ on $\Gamma$ (\cite{AM6},\cite{AM2},\cite{BR1})
 \item (B5) $\ub\times\nb=\gb\times\nb, p=\tilde{p}$ on $\Gamma$ (\cite{AM5},\cite{AM2})
 \item (B6) $\Gamma=\bar{\Gamma}_{1}\cup\bar{\Gamma}_{2}\cup\bar{\Gamma}_{3}$
            \begin{align}
            \begin{cases}
            \ub=\ub_{0}\quad\!\!\mbox{on}\quad\!\!\Gamma_{1},\\
            \ub\times\nb=\ab\times\nb, p=p_{0}\quad\!\!\mbox{on}\quad\!\!\Gamma_{2},\\
            \ub.\nb=\bb.\nb,(\nabla\times\ub)\times\nb=\hb\times\nb\quad\!\!\mbox{on}\quad\!\!\Gamma_{3}
            \end{cases}(\cite{CO1},\cite{BE1},\cite{CO2})\notag
            \end{align}
 \item (B7) $\ub.\nb=g,[(2\mathbb{D}\ub)\nb]_{\taub}=\hb\quad\!\!\mbox{on}\quad\!           
 \!\Gamma$ (\cite{AC1}, \cite{AM5})
 \item  (B8) $\ub.\nb=0, (\nabla\times\ub)_{\taub}=0\quad\!\!\mbox{on}\quad\!\!\Gamma$ (\cite{BE3})
 \item (B9)  $\Gamma=\bar{\Gamma}_{1}\cup\bar{\Gamma}_{2}$
             \begin{align}
             \begin{cases}
             \ub=\gb\quad\!\!\!\mbox{on}\quad\!\!\!\Gamma_{1},\\
              p=\phi,\ub\times\nb=\tilde{g}\quad\!\!\!\mbox{on}\quad\!\!\!\Gamma_{2}
              \end{cases}(\cite{BE2})\notag
              \end{align}
 \item (B10) $\big((\nu\nabla\ub-pI)\nb\big).\nb=\gb \quad\!\!\!\mbox{on}\quad\!\!\!\Gamma,
  \ub_{\taub}=0\quad\!\!\!\mbox{on}\quad\!\!\!\Gamma$ (\cite{BA1})
 \item  (B11) $\ub.\nb=0,\nabla\times\ub=0\quad\!\!\!\mbox{on}\quad\!\!\!\Gamma$ (\cite{OL1})
 \item (B12) $\Gamma=\bar{\Gamma}_{1}\cup\bar{\Gamma}_{2}\cup\bar{\Gamma}_{3}$
           \begin{align} 
            \begin{cases}
              \ub=\ub_{0}=(u_{01},u_{02})\quad\!\!\!\mbox{on}\quad\!\!\!\Gamma_{1},\\
              \ub\times\nb=u_{02},p=\phi\quad\!\!\!\mbox{on}\quad\!\!\!\Gamma_{2},\\
              \ub.\nb=u_{01},(\nabla\times\ub).\taub=\hb \quad\!\!\!\mbox{on}\quad\!\!\!\Gamma_{3}
            \end{cases} \cite{BE4}\notag
            \end{align}
  \item (B13) $\Gamma=\bar{\Gamma}_{1}\cup\bar{\Gamma}_{2}\cup\bar{\Gamma}_{3}$
               \begin{align}
               \begin{cases}
               \ub=\gb\quad\!\!\!\mbox{on}\quad\!\!\!\Gamma_{1},\\
               \nb.\ub=\gb_{n}, \nb\times((\nabla\times\ub)\times\nb)=\wa_{s}\quad\!\!\!\mbox{on}\quad\!\!\!\Gamma_{2},\\
               p=\psi, \nb\times(\ub\times\nb)=\gb_{s}\quad\!\!\!\mbox{on}\quad\!\!\!\Gamma_{3},\\
               p=\psi, \nb\times((\nabla\times \ub)\times\nb)=\gb_{s}\quad\!\!\!\mbox{on}\quad\!\!\!\Gamma_{4}
               \end{cases}(\cite{HU1})\notag
              \end{align}          
  \item (B14) $\Gamma=\bar{\Gamma}_{0}\cup\bar{\Gamma}_{1}$
              \begin{align}
               \begin{cases}
                \ub=0\quad\!\!\!\mbox{on}\quad\!\!\!\Gamma_{0},\\
                \ub_{\taub}=0,\sigma_{\nb}=\omega_{\nb}\quad\!\!\!\mbox{on}\quad\!\!\!\Gamma_{1}
                \end{cases}(\cite{SA1})\notag
               \end{align} 
  \item (B15) $\Gamma=\bar{\Gamma}_{0}\cup\bar{\Gamma}_{1}$
              \begin{align}
               \begin{cases}
               \ub=0\quad\!\!\!\mbox{on}\quad\!\!\!\Gamma_{0},\\
               \ub_{\nb}=0,\sigma_{\taub}=\omega_{\taub}\quad\!\!\!\mbox{on}\quad\!\!\!\Gamma_{1}
               \end{cases}(\cite{SA1})\notag
               \end{align}
  \item (B16) $\ub.\nb=g,\nu{\ub}_{\taub}+{[\sigma(\ub,p).\nb]}_{\taub}=\sbb\quad\!\!\!\mbox{on}\quad\!\!\!\Gamma$ (\cite{RU1}, \cite{VE1})
  \item (B17) $\ub.\taub=g, p=h \quad\!\!\!\mbox{on}\quad\!\!\!\Gamma$ (\cite{ME1})
  \item (B18) $\ub.\nb=g, {[(2\mathbb{D}\ub)\nb]}_{\taub}+\alpha{\ub}_{\taub}=\hb\quad\!\!\!\mbox{on}\quad\!\!\!{\Gamma}$ 
            (\cite{AM3},\cite{TA1})
\end{itemize}
\begin{rem}
 Here we have not mentioned the regularity of the domains in detail. However, information related to each boundary condition can be traced from the given references. For proposing the numerical schemes, we assume that the boundary of computational domains have the necessary regularity.
\end{rem}
\subsection{Motivations and physical applications }
Here we discuss physical applications of some of the boundary conditions.
 Stokes equations with boundary conditions of type (B1) are encountered in many physical problems such as a tank closed with a membrane on a part of the boundary, coupling with different equations such as Darcy's equations in the case of fluid domain is with a crack in a porus medium 
(\cite{BE6},\cite{BE7}). Use of (B4) in Stokes equations help to design variational formulations without $p$ term. The boundary conditions arise in fluid dynamics related applications, electromagnetic field applications, decomposition of vector fields 
(\cite{BE8}, \cite{BE9}, \cite{DO1}, \cite{GI1},\cite{VE3}).
Stokes equations with boundary conditions on tangential component of velocity and pressure boundary conditions (B5) are encountered in many physical applications such as Stokes equations in pipelines \cite{CO1} and blood vessels 
(\cite{VI1},\cite{VI2},\cite{TO2},\cite{OH1},\cite{GO1}). Boundary conditions of type (B6) are encountered in case of flow network in pipes, obstacles in pipes, modeling boundary conditions for infinite flow around obstacle etc. \cite{CO2}. Stokes equations with boundary conditions of type (B7) are used in understanding the coupling problem of Stokes and Darcy equations, which has immense importance in geophysical systems. The difficulty arises because of the large difference in scales between the porous media and the cracks where the flow is quite faster. Apart from cracks, such coupling problems have applications in seepage of water in sand e.g. from lake to ground, from sea to sands in the beach. Stokes equations with boundary conditions ((B2), (B3), (B16), (B18)) are called  slip boundary conditions \cite{VE2} and are used in modelling of flow 
problems with free boundaries ( e.g. coating problem \cite{SA2}), flows past chemically reacting walls and flow problems where no-slip boundary conditions
are not valid, incompressible viscous flows with high angles of attack and high Reynold's numbers. In the case of free surface problems \cite{SI1}, use of no-slip condition on the fixed part of the boundary leads to stress singularities. This is regarded as a case of nonphysical singularity and disappears with use of slip boundary conditions near the solid free surface contact point. Stokes equations with boundary conditions of type (B9) can be treated as generalization of Poiseuille flow in a channel, where $\Gamma_1$ represents the rigid wall and inflow/outflow takes place through $\Gamma_2$. Boundary conditions of type (B14) refers to fluid under investigation does not slip at the boundary and penetration of fluid through the boundary is controlled by the normal component of stress.
\subsection{Theoretical developments on Stokes equations with non-standard boundary conditions}
Next, we briefly discuss existing theoretical results on  Stokes equations with various
non-standard boundary conditions. These results help us in formulating numerical schemes in appropriate 
functional spaces.
Medkova \cite{ME1} has studied Stokes equations with three different boundary conditions (B2), (B3), (B4)
named as of Navier type on two dimensional bounded domains with Lipschitz boundary.
 Necessary and sufficient conditions for the existence of solutions on planar domains
in Sobolev spaces and Besov spaces have been obtained.
Amrouche et. al. \cite{AM2} have studied Stokes equations with boundary conditions (B4),(B5)
on three dimensional $C^{1,1}$ bounded multiply connected domains. Existence, uniqueness results of weak, strong and
very weak solutions have been provided. Amrouche et. al. \cite{AM3} have theoretically studied Stokes equations with 
boundary conditions (B18) with zero boundary data. Existence, uniqueness and regularity estimates on different Sobolev spaces have been provided.
Abboud et. al. (\cite{AB1},\cite{AB2}) have studied Stokes equations with boundary conditions (B4) and (B5)
with homogeneous boundary data
on three dimensional bounded, simply connected domains. Proposed variational formulation decouples the original problem into a system of
velocity and a Poisson problem for pressure. A priori and a posteriori estimates have been obtained and confirmed with numerical results.
A finite element formulation has been used for theoretical and numerical results.
Bernard \cite{BE1} has analyzed Stokes equations and Navier-Stokes equations with boundary conditions
(B6) for three dimensional bounded domains with boundary  of class $C^{1,1}.$ Existence, uniqueness, and regularity estimates
have been obtained in various Sobolev spaces. Results obtained are applicable for two dimensional domains too, with suitable modification.
Acvedo et. al. \cite{AC1} have studied Stokes operator  with boundary conditions (B7)
on three dimensional bounded $C^{1,1}$ domains. Existence, uniqueness and regularity results
have been obtained in different Sobolev spaces.
Bramble and Lee \cite{BR1} have investigated Stokes equations with boundary conditions (B4)
on three dimensional bounded domains. The existence, uniqueness results and a general shift theorem have been
provided. Existence and uniqueness results of variational formulation for Stokes equations with boundary
condition (B6) have been discussed in \cite{CO2}. Medkova \cite{ME2} has discussed solvability issues of the Stokes
system with (B17) in Sobolev and Besov spaces.
\subsection{Computational developments on Stokes equations with non-standard boundary conditions}
A spectral approximation of Stokes equations with boundary condition (B9)
has been investigated in \cite{BE2}. A stabilized hybrid discontinuous Galerkin approach using finite elements
has been used for Stokes equations with boundary conditions (B10) \cite{BA1}.
Bernardi and Chorfi \cite{BE3} has discussed a spectral formulation for vorticity based first order
formulation of Stokes equations with boundary condition (B8).
Numerical results suppporting
theoretical estimates have been presented. Bernard \cite{BE4} has presented a spectral discretization
(without numerical results) for Stokes equations with boundary conditions (B12).
Amara et. al. \cite{AM4} have discussed vorticity-velocity-pressure formulation for Stokes problem with 
boundary conditions (B13)
where $w=\nabla\times\ub, \Gamma={\overline{\Gamma}}_{1}\cup{\overline{\Gamma}}_{2}\cup{\overline{\Gamma}}_{3}$ on
two dimensional domains with assumptions that there is no non-convex corner at the intersection of 
${\overline{\Gamma}}_{1}\cup{\overline{\Gamma}}_{2}$ and ${\overline{\Gamma}}_{3}$. Conca et. al. \cite{CO1} have
presented a variational formulation for Stokes equations with boundary condition (B6). Authors have extended results
for Navier-Stokes equations. Optimal, a priori error estimates for 
Stokes equations with boundary conditions (B4) and (B5) with homogeneous data on three dimensional domains have been presented 
in \cite{AB1}. Hughes \cite{HU1}
has presented convergent symmetric finite element formulations for Stokes equations with (B13). Authors \cite{BE5} have discussed
Stokes equations with pressure boundary condition, (B5)  on two and three dimensional domains. Convergence analysis and numerical simulations have been presented. A spectral discretization for vorticity based 
first order formulation of Stokes equations with (B4) has been discussed in \cite{AM7}.
\begin{rem}
Development of computational aspects of Stokes equations with various non-standard boundary conditions is displayed in Table 1. Following table says no unified approach exists in the literature to accomodate so many boundary conditions. Also least-squares formulations have not been used. Hence least-squares approach to deal with various non-standard boundary condtions in a unified framework is the key achievement of this paper.
\begin{table}[H]
~~~~~~~~~~~~~~~~~~~~~~~~~~~%
\begin{center}
\begin{tabular}{l ccccccc}
\hline 
B.C.        & Method              & Numerical results & 2D/3D     & Extn to NS  & Ref\\
\hline 
(B4)/(B5)   &  Finite element     & Yes               & 3D        & No         &  \cite{AB1}\\
(B4)        &  Spectral Galerkin  & Yes               & 2D$\&$ 3D & No         &  \cite{AM7}\\
(B5)        & Finite element      & Yes               & 2D$\&$ 3D & No         &  \cite{BE5}\\
(B6)        & Finite element      & Yes               & 2D$\&$ 3D & Yes        &  \cite{CO1}\\
(B8)        & Spectral Galerkin   & Yes               & 2D$\&$ 3D & No         &  \cite{BE3}\\
(B9)        & Spectral Galerkin   & No                & 2D        & No         &  \cite{BE2}\\
(B10)       & Stabilized DGFEM    & Yes               & 2D        & No         &  \cite{BA1}\\
(B12)       & Spectral Galerkin   & No                & 2D        & No         &  \cite{BE4}\\
(B13)       & Finite element      & Yes               & 2D        & No         &  \cite{AM4}\\
(B13)       & Finite element      & No                & 2D$\&$ 3D & No         &  \cite{HU1}\\
\hline 
\end{tabular}
\end{center}
\caption{Summary of computational developments}
\end{table}
\end{rem}
\subsection{Regularity for $(\ub,p)$ in Stokes equations with various non-stanadard boundary conditions}
 We need regularity estimates of $\Hb^{2}$ and $H^{1}$ type for velocity and pressure variable respectively to propose the numerical schemes. 
In this section we discuss regularity estimates for Stokes equations
with different boundary conditions.
\begin{itemize}
\item Regularity for Stokes equations with (B4) \cite{AM2}\\
If $\fb\in\Lb^{2}(\Omega), g\in H^{\frac{3}{2}}(\Gamma),\hb\in\Hb^{\frac{1}{2}}(\Gamma)$, the solution
$(\ub,p)\in\Hb^{2}(\Omega)\times H^{1}(\Omega)$ and satisfies
\begin{align}
\|\ub\|_{\Hb^{2}(\Omega)}+\|p\|_{H^{1}(\Omega)}\leq c\big(\|\fb\|_{\Lb^{2}(\Omega)}+\|g\|_{H^{\frac{3}{2}}(\Gamma)}+\|\hb\times\nb\|_{\Hb^{\frac{1}{2}}(\Gamma)}\big).\notag
\end{align}
\item Regularity for Stokes equations with (B5) \cite{AM2}\\      
  If $\fb\in\Lb^{2}(\Omega), \gb\in \Hb^{\frac{3}{2}}(\Gamma),\tilde{p}\in H^{\frac{1}{2}}(\Gamma)$, the solution
$(\ub,p)\in\Hb^{2}(\Omega)\times H^{1}(\Omega)$ and satisfies
\begin{align}
\|\ub\|_{\Hb^{2}(\Omega)}+\|p\|_{H^{1}(\Omega)}\leq c\big(\|\fb\|_{\Lb^{2}(\Omega)}
+\|\gb\times\nb\|_{\Hb^{\frac{1}{2}}(\Gamma)}\notag
+\|\tilde{p}\|_{H^{\frac{1}{2}}(\Gamma)}\big).\notag
\end{align}  
\item  Regularity for Stokes equations with (B6) \cite{BE1}\\
      If $\fb\in\Lb^{2}(\Omega),\ub_{0}\in\Hb^{\frac{3}{2}}(\Gamma_{1}),
      \ab\in\Hb^{\frac{3}{2}}(\Gamma_{2}), p_{0}\in H^{\frac{1}{2}}(\Gamma_{2}),$ 
      $\bb\in\Hb^{\frac{3}{2}}(\Gamma_{3}), \hb\in\Hb^{\frac{1}{2}}(\Gamma_{3})$
    then Stokes equations with boundary conditions (B6)
    has a solution of the form $(\ub+\wa,p)$ with $\wa\in V_{1}$, where $V_{1}=\left\{\vb\in V : \nabla\times\vb=0\right\}$,  satisfying the regularity estimate
    \begin{align}
     \inf_{\wa\in V_{1}}\|\ub+\wa\|_{\Hb^{2}(\Omega)}+\|p\|_{H^{1}(\Omega)}
     \leq c\big(\|\fb\|_{\Lb^{2}(\Omega)}+\|\ub_{0}\|_{\Hb^{\frac{3}{2}}(\Gamma_{1})}+\|\ab\|_{\Hb^{\frac{3}{2}}(\Gamma_{1})}\big).\notag
    \end{align}
\item  Regularity for Stokes equations with (B7) \cite{AM5}\\
             If $\fb\in\Lb^{p}(\Omega),\chi\in H^{1}(\Omega),g\in H^{\frac{3}{2}}(\Gamma),\hb\in \Hb^{\frac{1}{2}}(\Gamma)$
             satisfy compatibility conditions (Theorem 4.1, \cite{AM5}), 
             then Stokes equations with boundary conditions (B7) has a unique solution
             $(\ub,p)\in (\Hb^{2}(\Omega)\times H^{1}(\Omega))/\mathcal{N}(\Omega)$ satisfying
             \begin{align}
              \|\ub\|_{\Hb^{2}(\Omega)/\mathcal{T}(\Omega)}+\|p\|_{H^{1}(\Omega)/\mathbb{R}}
              \leq c\big(\|\fb\|_{\Lb^{2}(\Omega)}+\|g\|_{H^{\frac{3}{2}}(\Gamma)}+\|\hb\|_{\Hb^{\frac{1}{2}}(\Gamma)}\big),\notag
             \end{align}
             where
             \begin{align}
              &\mathcal{T}(\Omega)=\left\{\ub\in\Hb^{1}(\Omega):\mathbb{D}\ub=0\quad\!\!\!\mbox{in}\quad\!\!\!\Omega,\nabla.\ub=0\quad\!\!\!\mbox{in}\quad\!\!\!\Omega, \ub.\nb=0 \quad\!\!\!\mbox{on}
              \quad\!\!\!\Gamma\right\}\notag\\
              &\mathcal{N}(\Omega)=\left\{(\ub,c): \ub\in\mathcal{T}(\Omega), c\in\mathbb{R}\right\}.\notag
             \end{align}
\item  Regularity for Stokes equations with (B11) \cite{OL1}\\
If $\fb\in\Lb^{2}(\Omega),\chi\in H^{1}(\Omega)\cap L^{2}/(\mathbb{R})$, then Stokes equations with (B11) has a
unique solution $(\ub,p)\in\Hb^{2}(\Omega)\times H^{1}(\Omega)$ satisfying
\begin{align}
 \|\ub\|_{\Hb^{2}(\Omega)}+\|p\|_{H^{1}(\Omega)}\leq c\big(\|\fb\|_{\Lb^{2}(\Omega)}+\|\chi\|_{\Hb^{1}(\Omega)}\big).\notag
\end{align}
\item Regularity for Stokes equations with (B14) \cite{SA1}\\
 For $\fb\in\Lb^{2}(\Omega), \omega_{\nb}\in\Hb^{\frac{1}{2}}(\Gamma)$.
 If $(\ub,p)$ are the weak solutions, then $\ub\in\Hb^{2}(\Omega), p\in H^{1}(\Omega)$ satisfying
 the inequality
 \begin{align}
  \|\ub\|_{\Hb^{2}(\Omega)}+\|p\|_{H^{1}(\Omega)}
  \leq c\big(\|\fb\|_{\Lb^{2}(\Omega)}+\|\omega_{\nb}\|_{\frac{1}{2},\Gamma}\big).\notag
 \end{align}
\item Regularity for Stokes equations with (B15) \cite{SA1}\\
For $\fb\in\Lb^{2}(\Omega), \omega_{\taub}\in\Hb^{\frac{1}{2}}(\Gamma)$,
we have $\ub\in\Hb^{2}(\Omega), p\in H^{1}(\Omega)$ satisfying
 the inequality
 \begin{align}
  \|\ub\|_{\Hb^{2}(\Omega)}+\|p\|_{H^{1}(\Omega)}\leq c\big(\|\fb\|_{\Lb^{2}(\Omega)}+\|\omega_{\taub}\|_{\frac{1}{2},\Gamma}\big).\notag
 \end{align}
\item Regularity of Stokes equations with boundary conditions (B18) \cite{TA1}

For $\fb\in \Lb^{2}(\Omega),\hb\in\Hb^{\frac{1}{2}}(\Gamma), (\ub,p)\in\Hb^{2}(\Omega)\times H^{1}(\Omega)$ 
      \begin{align}
       \|\ub\|_{\Hb^{2}(\Omega)}+\|p\|_{H^{1}(\Omega)}
       \leq c\big(\|\fb\|_{\Lb^{2}(\Omega)}+\|\hb\|_{\Hb^{\frac{1}{2}}(\Gamma)}\big).\notag
      \end{align}
\end{itemize}
\begin{rem}
 Similar regularity estimates are missing in the literature for some of the boundary conditions. However, we assume that boundary
 data of Dirichlet type for velocity is in $\Hb^{\frac{3}{2}}(\Gamma)$ and pressure is in $H^{\frac{1}{2}}(\Gamma)$.
\end{rem}
\section{Stability estimates}
In this section, we prove the stability estimates for Stokes problem with (B4) and (B5). Stability estimates in other boundary cases can be derived in a similar way. 
We denote $\xb=(x_{1},x_{2})$ and $\xb=(x_{1},x_{2},x_{3})$ if 
$\xb\in \Omega\subset\mathbb{R}^{n}, n=2,3$ respectively. $\Omega$ is divided into $L$ number of subdomains. For simplicity, 
$\Omega_{l}$'s are chosen to be rectangles in $\mathbb{R}^{2}$ and cubes in 
$\mathbb{R}^{3}$ respectively.
Let $Q$ denote the master element $Q=(-1,1)^{n}, n=2,3$. Now there is an analytic
map $M_l$ from $\Omega_l$ to $Q$ which has an analytic inverse. The map $M_l$ is of the form
$\hat\xb=M_l({\xb})$, where $\hat{\xb}=(\xi,\eta)$ and $\hat{\xb}=(\xi,\eta,\nu)$ if $\Omega\subset\mathbb{R}^{d}, d=2,3$
respectively.
Define the non-conforming spectral element functions $\ub_l$ and $p_l$ on $Q$ by
\begin{itemize}
 \item {if $\Omega_{l}\subseteq\Omega\subset {\mathbb{R}}^2$}
\begin{align}
\hat{\ub}|_Q(\xi,\eta)=\sum_{i=0}^W \sum_{j=0}^{W} 
\mathbf{a}_{i,j}\:\xi^i\eta^j,\quad \hat{p}|_{Q}(\xi,\eta)=\sum_{i=0}^{W}\sum_{j=0}^{W}b_{i,j}\xi^{i}\eta^{j}.\notag
\end{align}
\item {if $\Omega_{l}\subseteq\Omega\subset {\mathbb{R}}^3$}
\begin{align}
\hat{\ub}|_Q(\xi,\eta,\nu)=\sum_{i=0}^W \sum_{j=0}^W \sum_{k=0}^W \mathbf{a}_{i,j,k}\:\xi^i\eta^j\nu^k,
\quad \hat{p}|_{Q}(\xi,\eta,\nu)=\sum_{i=0}^{W}\sum_{j=0}^{W}\sum_{k=0}^{W}b_{i,j,k}\xi^{i}\eta^{j}\nu^{k}.\notag
\end{align}
\end{itemize}
Let $\Pib^{L,W}=\left\{\left\{\ub_{l}\right\}_{1\leq l\leq L},
\left\{p_{l}\right\}_{1\leq l\leq L}\right\}$  be the space of spectral element functions consisting of the above tensor
products of polynomials of degree $W$. 

Now let's define jump terms. Let
$\Gamma_{l,i}=\Gamma_{i}\cap\partial \Omega_l$ ($i^{th}$ interelement boundary of the $l^{th}$ element) 
be the image of the mapping $M_l$ corresponding
to $\eta=1$.
Let the face $\Gamma_{l,i}=\Gamma_{m,j}$, where $\Gamma_{m}$ is a edge/face of the element $\Omega_m$ i.e.
$\Gamma_{l}$ is a edge/face common to the elements $\Omega_{l},\Omega_{m}$. We may assume the edge
$\Gamma_{l,i}$ corresponds to $\eta=1$ and $\Gamma_{m,j}$ corresponds to $\eta=-1$. Let $\left.{[\ub]}
\right|_{\Gamma_{l,i}}$ denote the jump in $\ub$ across the edge/face $\Gamma_{l,i}$. 
If $\Gamma_{l,i}\subset\partial\Omega_{l}\cap\partial\Omega_{m}$ when $\Omega_{l},\Omega_{m}\subset {\mathbb{R}}^{2}$, we define jump
terms as :
\begin{align}
  &\left\Vert [\hat\ub]\right\Vert _{0,\Gamma_{l,i}}^{2}=
  \left\Vert \hat{\ub}_{m}(\xi,-1)-\hat{\ub}_{l}(\xi,1)\right\Vert _{0,\Gamma_{l,i}}^{2}\notag\\
  &\left\Vert [\hat\ub_{x_{k}}]\right\Vert _{\frac{1}{2},\Gamma_{l,i}}^{2}=
  \| ({\hat{\ub}_{m})}_{x_{k}}(\xi,-1)
              -{(\hat{\ub}_{l})}_{x_{k}}(\xi,1)
                \|_{\frac{1}{2},\Gamma_{l,i}}^{2}\notag\\
  &\left\Vert [\hat p]\right\Vert _{\frac{1}{2},\Gamma_{l,i}}^{2}=
  \left\Vert\hat{ p}_{m}(\xi,-1)-\hat{p}_{l}(\xi,1)\right\Vert_{\frac{1}{2},\Gamma_{l,i}}^{2}.\notag
\end{align}
Jumps along faces in case of three dimensional domains can be defined in a similar way.

Let $\ub,p\in\Pib^{L,W}$. 
Next, we define the quadratic form
\begin{align}\label{precon}
\mathcal U^{L,W}(\ub,p)=\sum_{l=1}^{L}\|\ub_{l}\|_{2,Q}^{2}+\sum_{l=1}^{L}\|p_{l}\|_{1,Q}^{2}.
\end{align}
\subsection{Stability estimate for Stokes equations with boundary condition (B4)}
We now define the quadratic form
\begin{align}
\mathcal V^{L,W}(\ub,p)&=\sum_{l=1}^{L}\|\mathcal{L}(\ub_l,p_l)\|_{0,Q}^{2}
+\sum_{l=1}^{L}\|\mathcal{D}\ub_l\|_{1,Q}^{2}\notag\\
&+\sum_{\Gamma_{l,i}\subseteq \bar{\Omega}\setminus \Gamma}\left(\|[\ub]\|_{0,\Gamma_{l,i}}^{2}
+\sum_{k=1}^{n}\|[\ub_{x_{k}}]\|_{\frac{1}{2},\Gamma_{l,i}}^{2}+\|[p]\|_{\frac{1}{2},\Gamma_{l,i}}^{2}\right)\notag\\
&+\sum_{\Gamma_{l,i}\subset\Gamma}\|\ub_l\cdot\nb\|_{\frac{3}{2},\Gamma_{l,i}}^{2}
+\sum_{\Gamma_{l,i}\subset\Gamma}\|(\nabla\times\ub_{l})\times\nb\|_{\frac{1}{2},\Gamma_{l,i}}^{2}.\notag
\end{align}
\begin{thm}\label{thm1}
Consider the Stokes equations (\ref{eq1}). Then for $W$ large enough there exists a constant $c>0$ (independent of $W$) such
that the estimate
\begin{align}\label{sb1.a}
\mathcal U^{L,W}(\ub,p)\leq c(\ln W)^{2}\mathcal V^{L,W}(\ub,p)
\end{align}
holds.
\end{thm}
\begin{proof}
By Theorem 2.3 of \cite{AM} we have,
\begin{align}
\left\Vert \ub\right\Vert_{\Hb^{2}(\Omega)}^{2}+ \left\Vert p\right\Vert _{H^{1}(\Omega)}^{2}
\leq c\left(\left\Vert \mathcal{L}(\ub,p)\right\Vert_{_{\Lb^{2}(\Omega)}}^{2}
+\left\Vert \mathcal{D}\ub\right\Vert_{{H^{1}(\Omega)}}^{2}
+\left\Vert \ub\cdot\nb\right\Vert_{{H^{\frac{3}{2}}(\Gamma)}}^{2}
+{\|(\nabla\times\ub)\times\nb\|}_{\Hb^{\frac{1}{2}}(\Gamma)}
\right).\notag
\end{align}
We have, 
\begin{align}
\left\Vert \mathcal{L}(\ub,p)\right\Vert _{\Lb^{2}(\Omega)}^{2}
\leq  c\sum_{l=1}^{L}\left\Vert\mathcal{L}_{l}(\ub_{l},p_{l})
\right\Vert_{_{0,Q}}^{2}\quad\!\!\! 
\mbox{and}\quad\!\!\!
\left\Vert \mathcal{D}\ub\right\Vert_{H^{1}(\Omega)}^{2}
\leq c\sum_{l=1}^{L}\left\Vert\mathcal{D}_{l}\ub_{l}\right\Vert_{_{1,Q}}^{2}.\notag
\end{align}
Also we have,
\begin{align}
&{\|\ub\cdot\nb\|}_{\Hb^{\frac{3}{2}}(\Gamma)}^{2}+{\|\nabla\times\ub\times\nb\|}_{\Hb^{\frac{1}{2}}(\Gamma)}^{2}\notag\\
&\leq c\bigg(\left\Vert\ub\right\Vert_{\Hb^{\frac{3}{2}}(\Gamma)}^{2}+\|\frac{\partial\ub}{\partial\nb}\|_{\Hb^{\frac{1}{2}}(\Gamma)}\bigg)\notag\\
&\leq  c(\ln W)^{2}\left(\sum_{\Gamma_{l,i}\subseteq\bar{\Omega}\setminus\Gamma}\left(\left\Vert[\ub]
\right\Vert_{_{0,\Gamma_{l,i}}}^{2}+\sum_{k=1}^{n}\left\Vert[\ub_{x_{k}}]\right\Vert_{_{1/2,\Gamma_{l,i}}}^{2}\right)
+\sum_{\Gamma_{l,i}\subset\Gamma}\left\Vert \ub\right\Vert_{_{\frac{3}{2},\Gamma_{l,i}}}^{2}\right).\notag\\
&\mbox{(Using estimates and techniques from~\cite{KI2}, \cite{SM3})}\notag
\end{align}
Combining above estimates, we obtain
\begin{align}
\left\Vert \ub \right\Vert_{\Hb^{2}(\Omega)}^{2}+\left\Vert p \right\Vert_{H^{1}(\Omega)}^{2}
\leq c (\ln W)^{2}\mathcal{V}^{^{L,W}}(\ub,p).\notag
\end{align}
From this (\ref{sb1.a}) follows.
\end{proof}
\subsubsection{Stability estimate for Stokes equations with boundary condition (B5)}
We now define the quadratic form
\begin{align}
\mathcal V^{L,W}(\ub,p)&=\sum_{l=1}^{L}\|\mathcal{L}(\ub_l,p_l)\|_{0,Q}^{2}
+\sum_{l=1}^{L}\|\mathcal{D}\ub_l\|_{1,Q}^{2}\notag\\
&+\sum_{\Gamma_{l,i}\subseteq \bar{\Omega}\setminus \Gamma}\left(\|[\ub]\|_{0,\Gamma_{l,i}}^{2}
+\sum_{k=1}^{n}\|[\ub_{x_{k}}]\|_{\frac{1}{2},\Gamma_{l,i}}^{2}+\|[p]\|_{\frac{1}{2},\Gamma_{l,i}}^{2}\right)\notag\\
&+\sum_{\Gamma_{l,i}\subset\Gamma}\|\ub_l\times\nb\|_{\frac{3}{2},\Gamma_{l,i}}^{2}
+\sum_{\Gamma_{l,i}\subset\Gamma}\|p_{l}\|_{\frac{1}{2},\Gamma_{l,i}}^{2}\notag
\end{align}
\begin{thm}\label{thm2}
Consider the Stokes equations (\ref{eq1}). Then for $W$ large enough there exists a constant $c>0$ (independent of $W$) such
that the estimate
\begin{align}\label{sb1}
\mathcal U^{L,W}(\ub,p)\leq c(\ln W)^{2}\mathcal V^{L,W}(\ub,p)
\end{align}
holds.
\end{thm}
\begin{proof}
By Theorem 2.3 of \cite{AM} we have,
\begin{align}
\left\Vert \ub\right\Vert_{\Hb^{2}(\Omega)}^{2}+ \left\Vert p\right\Vert _{H^{1}(\Omega)}^{2}
\leq c\left(\left\Vert \mathcal{L}(\ub,p)\right\Vert_{_{\Lb^{2}(\Omega)}}^{2}
+\left\Vert \mathcal{D}\ub\right\Vert_{{H^{1}(\Omega)}}^{2}
+\left\Vert \ub\times\nb\right\Vert_{{\Hb^{\frac{3}{2}}(\Gamma)}}^{2}
+{\|p\|}_{H^{\frac{1}{2}}(\Gamma)}
\right).\notag
\end{align}
Now,
\begin{align}
\left\Vert \mathcal{L}(\ub,p)\right\Vert _{\Lb^{2}(\Omega)}^{2}
\leq  c\sum_{l=1}^{L}\left\Vert\mathcal{L}_{l}(\ub_{l},p_{l})
\right\Vert_{_{0,Q}}^{2}\notag
\end{align}
and
\begin{align}
\left\Vert \mathcal{D}\ub\right\Vert_{H^{1}(\Omega)}^{2}
\leq c\sum_{l=1}^{L}\left\Vert\mathcal{D}_{l}\ub_{l}\right\Vert_{_{1,Q}}^{2}.\notag
\end{align}
Also we have,
\begin{align}
&{\|\ub\times\nb\|}_{\Hb^{\frac{3}{2}}(\Gamma)}^{2}+{\|p\|}_{H^{\frac{1}{2}}(\Gamma)}^{2}\notag\\
&\leq \left\Vert\ub\right\Vert_{\Hb^{\frac{3}{2}}(\Gamma)}^{2}+\left\Vert p \right\Vert_{H^{\frac{1}{2}}(\Gamma)}^{2}\notag\\
&\leq  c(\ln W)^{2}\left(\sum_{\Gamma_{l,i}\subseteq\bar{\Omega}\setminus\Gamma}\left(\left\Vert[\ub]
\right\Vert_{_{0,\Gamma_{l,i}}}^{2}+\sum_{k=1}^{n}\left\Vert[\ub_{x_{k}}]\right\Vert_{_{1/2,\Gamma_{l,i}}}^{2}\right)
+\sum_{\Gamma_{l,i}\subset\Gamma}\left\Vert \ub\right\Vert_{_{\frac{3}{2},\Gamma_{l,i}}}^{2}
+\left\Vert p\right\Vert_{_{\frac{1}{2},\Gamma_{l,i}}}^{2}\right).\notag\\
&\mbox{(Using estimates and techniques from~\cite{KI2}, \cite{SM3})}\notag
\end{align}
Combining above estimates, we obtain
\begin{align}
\left\Vert \ub \right\Vert_{\Hb^{2}(\Omega)}^{2}+\left\Vert p \right\Vert_{H^{1}(\Omega)}^{2}
\leq c(\ln W)^{2}\mathcal{V}^{^{L,W}}(\ub,p).\notag
\end{align}
From this (\ref{sb1}) follows.
\end{proof}
\section{Numerical scheme}
In this section, we describe numerical scheme for each boundary condition. Numerical schemes are based on a non-conforming least-squares formulation. Spectral element functions are allowed to be discontinuous along interelement boundaries. We present the numerical schemes as minimize
${\mathcal{R}}^{L,W}(\ub,p)={\mathcal{R}1}^{L,W}(\ub,p)+{\mathcal{R}2}^{L,W}(\ub,p)$, where
\begin{align}
{\mathcal{R}1}^{L,W}(\ub,p)&=\underbrace{\sum_{l=1}^{L}\|\mathcal{L}(\ub_l,p_l)-\Fb_{l}\|_{0,\Omega_l}^{2}}_{(a)}\notag\\
&+\underbrace{\sum_{l=1}^{L}\|\mathcal{D}\ub_l-\chi_{l}\|_{1,\Omega_l}^{2}}_{(b)}\notag\\
&+\underbrace{\sum_{\Gamma_{l}\subset \bar{\Omega}\setminus \Gamma}\left(\|[\ub]\|_{0,\Gamma_{l}}^{2}
+\sum_{k=1}^{n}\|[\ub_{x_{k}}]\|_{\frac{1}{2},\Gamma_{l}}^{2}+\|[p]\|_{\frac{1}{2},\Gamma_{l}}^{2}\right)}_{(c)}\notag
\end{align}
and remains unchanged for all cases. Here $(a)$ minimizes residual in momentum equations, 
$(b)$ minimizes residual in continuity equation and $(c)$ minimizes jump in unknowns and their derivatives across the interelement
boundaries.
\begin{itemize}
 \item ${\mathcal{R}2}^{L,W}(\ub,p)$ for boundary condition (B1)
\begin{align}
{\mathcal{R}2}^{L,W}(\ub,p)&=\sum_{\Gamma_{l}\subseteq \Gamma_{1}\cap\Gamma}\|\ub_{l}-\gb_{l}\|_{\frac{3}{2},\Gamma_{l}}^{2}
+\sum_{\Gamma_{l}\subseteq \Gamma_{2}\cap\Gamma}\|\ub_{l}.\nb-\gb_{\nb}\|_{\frac{3}{2},\Gamma_{l}}^{2}\notag\\
&+\sum_{\Gamma_{l}\subseteq \Gamma_{2} \cap \Gamma}{\|(\nabla\times\ub_{l})\times\nb-\kb\times\nb\|}_{\frac{1}{2},\Gamma_{l}}^{2}.\notag
\end{align}
\item ${\mathcal{R}2}^{L,W}(\ub,p)$ for boundary condition (B2)
\begin{align}
{\mathcal{R}2}^{L,W}(\ub,p)=
\sum_{\Gamma_{l}\subset\Gamma}\|(\ub_l)_{\nb}-(\gb_{l})_{\nb}\|_{\frac{3}{2},\Gamma_{l}}^{2}
+\sum_{\Gamma_{l}\subset\Gamma}
{\left\Vert\bigg[\frac{\partial\ub_{l}}{\partial\nb}-p_{l}\nb+b\ub_{l}\bigg]_{\tau}-(\hb_{l})_{\taub}\right\Vert}_{\frac{1}{2},\Gamma_{l}}^{2}.\notag
\end{align}
\item ${\mathcal{R}2}^{L,W}(\ub,p)$ for boundary condition (B3)
\begin{align}
{\mathcal{R}2}^{L,W}(\ub,p)=\sum_{\Gamma_{l}\subset\Gamma}\|(\ub_l)_{\nb}-(\gb_{l})_{\nb}\|_{\frac{3}{2},\Gamma_{l}}^{2}
+\sum_{\Gamma_{l}\subset\Gamma}
{\left\Vert\big([\nabla\ub_{l}+(\nabla\ub_{l})^{T}-pI]\nb+b\ub_{l}\big)_{\taub}-(\hb_{l})_{\taub}\right\Vert}_{\frac{1}{2},\Gamma_{l}}^{2}.\notag
\end{align}
\item ${\mathcal{R}2}^{L,W}(\ub,p)$ for boundary condition (B4)
\begin{align}
{\mathcal{R}2}^{L,W}(\ub,p)=
\sum_{\Gamma_{l}\subset\Gamma}\|\ub_l.\nb-\gb_{l}\|_{\frac{3}{2},\Gamma_{l}}^{2}
+\sum_{\Gamma_{l}\subset\Gamma}
{\left\Vert(\nabla\times\ub_{l}\times\nb)-\hb_{l}\times\nb\right\Vert}_{\frac{1}{2},\Gamma_{l}}^{2}.\notag
\end{align}
\item ${\mathcal{R}2}^{L,W}(\ub,p)$ for boundary condition (B5)
\begin{align}
{\mathcal{R}2}^{L,W}(\ub,p)=
\sum_{\Gamma_{l}\subset\Gamma}\|(\ub_l\times\nb)-(\gb_{l}\times\nb)\|_{\frac{3}{2},\Gamma_{l}}^{2}
+\sum_{\Gamma_{l}\subset\Gamma}
{\left\Vert(p_{l}-\tilde{p}_{l})\right\Vert}_{\frac{1}{2},\Gamma_{l}}^{2}.\notag
\end{align}
\item ${\mathcal{R}2}^{L,W}(\ub,p)$ for boundary condition (B6)
\begin{align}
{\mathcal{R}2}^{L,W}(\ub,p)&=
\sum_{\Gamma_{l}\subseteq \Gamma_{1}\cap\Gamma}\|\ub_l-\ub_{0}\|_{\frac{3}{2},\Gamma_{l}}^{2}
+\sum_{\Gamma_{l}\subseteq \Gamma_{2}\cap\Gamma}\|\ub_{l}\times\nb-\ab_{l}\times\nb\|_{\frac{3}{2},\Gamma_{l}}^{2}\notag\\
&+\sum_{\Gamma_{l}\subseteq \Gamma_{2}\cap\Gamma}\|p_{l}-p_{0}\|_{\frac{1}{2},\Gamma_{l}}^{2}
+\sum_{\Gamma_{l}\subseteq \Gamma_{3}\cap\Gamma}\|\ub_{l}.\nb-\bb_{l}.\nb\|_{\frac{3}{2},\Gamma_{l}}^{2}\notag\\
&+\sum_{\Gamma_{l}\subseteq \Gamma_{3} \cap \Gamma}\!\!
{\|(\nabla\times{\ub_{l}}\times\nb_{l})-\hb_{l}\times\nb\|}_{\frac{1}{2},\Gamma_{l}}^{2}.\notag
\end{align}
\item ${\mathcal{R}2}^{L,W}(\ub,p)$ for boundary condition (B7)
\begin{align}
{\mathcal{R}2}^{L,W}(\ub,p)=
\sum_{\Gamma_{l}\subset \Gamma}\!\!\!\|\ub_l.\nb-g_{l}\|_{\frac{3}{2},\Gamma_{l}}^{2}
+\!\!\!\sum_{\Gamma_{l}\subset \Gamma}\!\!\!{\|[(2\mathbb{D}\ub_{l})\nb]_{\taub}-\hb_{l}\|}_{\frac{1}{2},\Gamma_{l}}^{2}.\notag
\end{align}
\item ${\mathcal{R}2}^{L,W}(\ub,p)$ for boundary condition (B8)
\begin{align}
{\mathcal{R}2}^{L,W}(\ub,p)
=\sum_{\Gamma_{l}\subset\Gamma}\|\ub_l.\nb\|_{\frac{3}{2},\Gamma_{l}}^{2}
+\sum_{\Gamma_{l}\subset\Gamma}
{\left\Vert(\nabla\times\ub_{l})_{\taub})\right\Vert}_{\frac{1}{2},\Gamma_{l}}^{2}.\notag
\end{align}
\item  ${\mathcal{R}2}^{L,W}(\ub,p)$ for boundary condition (B9)
\begin{align}
{\mathcal{R}2}^{L,W}(\ub,p)&=
\sum_{\Gamma_{l}\subseteq \Gamma_{1}\cap\Gamma}\|\ub_{l}-\gb_{l}\|_{\frac{3}{2},\Gamma_{l}}^{2}
+\sum_{\Gamma_{l}\subseteq \Gamma_{2}\cap\Gamma}\|p_{l}-\phi_{l}\|_{\frac{1}{2},\Gamma_{l}}^{2}
+\sum_{\Gamma_{l}\subseteq \Gamma_{2}\cap\Gamma}\|\ub_{l}\times\nb-\gb_{l}\|_{\frac{3}{2},\Gamma_{l}}^{2}.\notag
\end{align}
\item ${\mathcal{R}2}^{L,W}(\ub,p)$ for boundary condition (B10)
\begin{align}
{\mathcal{R}2}^{L,W}(\ub,p)=
\sum_{\Gamma_{l}\subset \Gamma }
\|((\nu\nabla\ub_{l}-p_{l}I)\nb).\nb-\gb_{l}\|_{\frac{1}{2},\Gamma_{l}}^{2}
+\sum_{\Gamma_{l}\subset \Gamma}\|{(\ub_{l})}_{\taub}\|_{\frac{3}{2},\Gamma_{l}}^{2}.\notag
\end{align}
\item ${\mathcal{R}2}^{L,W}(\ub,p)$ for boundary condition (B11)
\begin{align}
{\mathcal{R}2}^{L,W}(\ub,p)=\sum_{\Gamma_{l}\subset\Gamma}
\|\ub_{l}.\nb\|_{\frac{3}{2},\Gamma_{l}}^{2}
+\sum_{\Gamma_{l}\subset \Gamma}
\|\nabla\times\ub_{l}\|_{\frac{1}{2},\Gamma_{l}}^{2}.\notag
\end{align}
\item ${\mathcal{R}2}^{L,W}(\ub,p)$ for boundary condition (B12)
\begin{align}
{\mathcal{R}2}^{L,W}(\ub,p)&=\sum_{\Gamma_{l}\subseteq \Gamma_{1}\cap\Gamma}\|\ub_l-\ub_{0}\|_{\frac{3}{2},\Gamma_{l}}^{2}
+\sum_{\Gamma_{l}\subseteq \Gamma_{2}\cap\Gamma}\|\ub_{l}\times\nb-u_{02}\|_{\frac{3}{2},\Gamma_{l}}^{2}\notag\\
&+\sum_{\Gamma_{l}\subseteq \Gamma_{2}\cap\Gamma}\|p_{l}-\phi_{l}\|_{\frac{1}{2},\Gamma_{l}}^{2}
+\sum_{\Gamma_{l}\subseteq \Gamma_{3}\cap\Gamma}\|\ub_{l}.\nb-u_{01}\|_{\frac{3}{2},\Gamma_{l}}^{2}\notag\\
&+\sum_{\Gamma_{l}\subseteq \Gamma_{3} \cap \Gamma}{\|(\nabla\times{\ub_{l}}.\taub)-\hb\|}_{\frac{1}{2},\Gamma_{l}}^{2}.\notag
\end{align}
\item ${\mathcal{R}2}^{L,W}(\ub,p)$ for boundary condition (B13)
\begin{align}
{\mathcal{R}2}^{L,W}(\ub,p)
&=\sum_{\Gamma_{l}\subseteq \Gamma_{1}\cap\Gamma}\|\ub_{l}-\gb_{l}\|_{\frac{3}{2},\Gamma_{l}}^{2}\notag\\
&+\sum_{\Gamma_{l}\subseteq \Gamma_{2}\cap\Gamma}\|\nb.\ub_{l}-\gb_{\nb}\|_{\frac{3}{2},\Gamma_{l}}^{2}
+\sum_{\Gamma_{l}\subseteq \Gamma_{2}\cap\Gamma}\|\nb\times(\nabla\times\ub_{l})\times\nb-\omegab_{s}\|_{\frac{1}{2},\Gamma_{l}}^{2}\notag\\
&+\sum_{\Gamma_{l}\subseteq \Gamma_{3}\cap\Gamma}\|p_{l}-\psi_{l}\|_{\frac{1}{2},\Gamma_{l}}^{2}
+\sum_{\Gamma_{l}\subseteq \Gamma_{3}\cap\Gamma}\|\nb\times(\ub_{l}\times\nb)-\gb_{s}\|_{\frac{3}{2},\Gamma_{l}}^{2}\notag\\
&+\sum_{\Gamma_{l}\subseteq \Gamma_{4}\cap\Gamma}\|p_{l}-\psi_{l}\|_{\frac{1}{2},\Gamma_{l}}^{2}
++\sum_{\Gamma_{l}\subseteq \Gamma_{4}\cap\Gamma}\|\nb\times(\nabla\times\ub_{l})\times\nb-\gb_{s}\|_{\frac{1}{2},\Gamma_{l}}^{2}.\notag
\end{align}
\item ${\mathcal{R}2}^{L,W}(\ub,p)$ for boundary condition (B14)
\begin{align}
{\mathcal{R}2}^{L,W}(\ub,p)=
\sum_{\Gamma_{l}\subseteq \Gamma_{0}\cap\Gamma}\|\ub_l\|_{\frac{3}{2},\Gamma_{l}}^{2}
+\sum_{\Gamma_{l}\subseteq \Gamma_{1} \cap \Gamma}
{\|{\ub_{l}}_{\taub}\|}_{\frac{3}{2},\Gamma_{l}}^{2}
+\sum_{\Gamma_{l}\subseteq \Gamma_{1} \cap \Gamma}
{\|\sigma_{\nb}-\omega_{\nb}\|}_{\frac{1}{2},\Gamma_{l}}^{2}.\notag
\end{align} 
\item ${\mathcal{R}2}^{L,W}(\ub,p)$ for boundary condition (B15)
\begin{align}
{\mathcal{R}2}^{L,W}(\ub,p)
=\sum_{\Gamma_{l}\subseteq \Gamma_{0}\cap\Gamma}\|\ub_l\|_{\frac{3}{2},\Gamma_{l}}^{2}
+\sum_{\Gamma_{l}\subseteq \Gamma_{1} \cap \Gamma}
{\|\ub_{\nb}\|}_{\frac{3}{2},\Gamma_{l}}^{2}
+\sum_{\Gamma_{l}\subseteq \Gamma_{1} \cap \Gamma}
{\|\sigma_{\taub}-\omega_{\taub}\|}_{\frac{1}{2},\Gamma_{l}}^{2}.\notag
\end{align} 
\item ${\mathcal{R}2}^{L,W}(\ub,p)$ for boundary condition (B16)
\begin{align}
{\mathcal{R}2}^{L,W}(\ub,p)&=
\sum_{\Gamma_{l}\subset \Gamma} \|\ub_l.\nb-\gb_{l}.\nb\|_{\frac{3}{2},\Gamma_{l}}^{2}
+\sum_{\Gamma_{l}\subseteq \Gamma}
\|{\nu(\ub_{l})_{\taub}+[\sigma(\ub_{l},p).\nb]}_{\taub}-\sbb\|_{\frac{1}{2},\Gamma_{l}}^{2}.\notag
\end{align} 
\item ${\mathcal{R}2}^{L,W}(\ub,p)$ for boundary condition (B17)
\begin{align}
{\mathcal{R}2}^{L,W}(\ub,p)
&=\sum_{\Gamma_{l}\subset \Gamma}\|\ub_l.\taub-g\|_{\frac{3}{2},\Gamma_{l}}^{2}
+\sum_{\Gamma_{l}\subset \Gamma}\|p_{l}-h_{l}\|_{\frac{1}{2},\Gamma_{l}}^{2}.\notag
\end{align}
\item ${\mathcal{R}2}^{L,W}(\ub,p)$ for boundary condition (B18)
\begin{align}
{\mathcal{R}2}^{L,W}(\ub,p)
&=\sum_{\Gamma_{l}\subseteq \Gamma}\|\ub_l.\nb-g\|_{\frac{3}{2},\Gamma_{l}}^{2}
+\sum_{\Gamma_{l}\subseteq \Gamma}\|[(2\mathbb{D}\ub_{l})\nb]_{\taub}+\alpha(\ub_{l})_{\taub}-\hb_{l}\|_{\frac{1}{2},\Gamma_{l}}^{2}.\notag
\end{align}
\end{itemize}
\begin{rem}
 The general idea behind proposing these above schemes for various boundary conditions is based on the construction of corresponding norm-equivalent least-squares functional which is similar to the idea behind formulations used in \cite{SM3}, \cite{SM1}, \cite{SM2}. The norms are chosen based on norms used in available regularity estimates in the literature.
\end{rem}
\begin{rem}
 Some of the least-squares formulations (\cite{CA1}, \cite{PR1}) use weights to continuity equation related term in the least squares functional to enhance mass conservation. However, we don't use any extra weights in the least squares functionals. Use of weights may lead to ill-conditioning. In the numerical results section, mass conservation for each case has been displayed.
\end{rem}
\section{Error estimates}
In this section, we present error estimates for two dimensional domains. The estimates for three dimensional case follows exactly in a similar way.
\begin{thm}\label{app312}
 Let $\ub_{l}(\xi,\eta)=\ub(M_{l}(\xi,\eta))$, $p_{l}(\xi,\eta)=p(M_{l}(\xi,\eta))$,
 for $(\xi,\eta)\in Q$.
 Then there exist positive constants $c$ and $b$ (both are independent of $W$) such that for $W$ large enough the estimate
\begin{align}
\sum_{l=1}^{L}{\|\zb_{l}-\ub_{l}\|}_{2,Q}^{2}
+\sum_{l=1}^{L}{\|q_{l}-p_{l}\|}_{1,Q}^{2}\leq ce^{-bW}\notag
\end{align} holds.
\end{thm}
\textbf{Proof :}
 From \cite{CH}, we can have a polynomial $\Phi(\xi,\eta)$ of degree $W$ in each variable separately such that
 \begin{align}
 {\|v(\xi,\eta)-\Phi(\xi,\eta)\|}_{n,S}^{2}\leq c_{s}W^{4+2n-2s}\|v\|_{s,S}^{2}\notag
 \end{align}
for $0\leq n\leq 2$ and all $W>s,$ where  $c_{s}=ce^{2s}.$

Hence there exist polynomials $\psy_{l}(\xi,\eta),\Phi_{l}(\xi,\eta)$, $0\leq
l\leq L$ such that
\begin{align}
 & \left\Vert \ub_{l}(\xi,\eta)-\psy_{l}(\xi,\eta)\right\Vert
_{2,Q}^{2}\leq c_{s}W^{-2s+8}(Cd^{s}s!)^{2},\notag\\
 &  \left\Vert p_{l}(\xi,\eta)-\phi_{l}(\xi,\eta)\right\Vert
_{1,Q}^{2}\leq c_{s}W^{-2s+6}(Cd^{s}s!)^{2}.\notag
\end{align} 
Consider the set of functions
$\Bigg\{
 \left\{ \psy_{l}(\xi,\eta)\right\},
 \left\{ \phi_{l}(\xi,\eta)\right\}
 \Bigg\}.$
 We wish to show
$\mathcal{R}^{L,W}\Bigg(
\left\{\psy_{l}(\xi,\eta)\right\},
\left\{\phi_{l}(\xi,\eta)\right\}\Bigg)\notag
$ is exponentially small.

Using concepts from trace theory, 
\begin{itemize}
\item for $u\in H^{2}(\Omega)$ we have
\begin{align}
 \|u\|_{\frac{3}{2},\Gamma}\leq c\|u\|_{2,\Omega}.\notag
\end{align}
\item for $u\in H^{1}(\Omega)$ we have
\begin{align}
 \|u\|_{\frac{1}{2},\Gamma}\leq c\|u\|_{1,\Omega}.\notag
\end{align}
\end{itemize}
Above inequalities can be used to estimate the boundary related terms in the least-squares functional.
We can show that
\begin{align}
&{\mathcal{R}}^{L,W}\Bigg(
\left\{\psy_{l}(\xi,\eta)\right\},
\left\{\phi_{l}(\xi,\eta)\right\}\Bigg)\notag\\
&\leq cL\Bigg(c_{s}W^{-2s+8}lnW (Cd^{s}s!)^{2}+\tilde{c}_{s}W^{-2s+6} lnW(Cd^{s}s!)^{2}\Bigg)\notag\\
&            +cL\Bigg(c_{t}(2W-1)^{-2t+8}lnW (Cd^{t}t!)^{2}+\tilde{c}_{t}W^{-2t+6} lnW(Cd^{t}t!)^{2}\Bigg).\notag
\end{align}
By using Sterling's formula and techniques from Theorem-3.1, \cite{TO}, we can see that
 there exists a constant $b>0$ such that the estimate
\begin{align}
\mathcal{R}^{L,W}\Bigg(
\!\left\{\psy_{l}(\xi,\eta)\right\},
\!\left\{\phi_{l}(\xi,\eta)\right\}\Bigg)
\leq
ce^{-bW}\notag
\end{align} holds.
Let
$(\!\left\{ \zb_{l}\right\},\!\left\{s_{l}\right\})$ minimize
${\mathcal{R}}^{L,W}(\left\{\ub_{l}\right\},\left\{p_{l}\right\})$
 over all
$\Bigg\{\left\{ \ub_{l}\right\},\left\{q_{l}\right\}\Bigg\}
\in\Pi^{M,W},$
the space of spectral element functions.
 Then we have
\begin{align}
{\mathcal{R}}^{L,W}\big(
\!\left\{\zb_{l}\right\},
\!\left\{s_{l}\right\}\big)
\leq ce^{-bW}.\notag
\end{align}
Therefore, we can conclude that
\begin{align}
\mathcal{V}^{L,W}\Bigg(
\left\{\psy_{l}-\zb_{l}\right\},
 \left\{\Phi_{l}-s_{l}\right\}\Bigg)
\leq ce^{-bW}\notag
\end{align}
  where the functional
$\mathcal{V}^{L,W}$ equals ${\mathcal{R}}^{L,W}$ with zero right hand side data.
Hence using the stability Theorems \ref{thm1}, \ref{thm2}  (similar theorems for different boundary conditions) we obtain
\begin{align}
\sum_{l=1}^{L}\left\Vert
\psy_{l}-\zb_{l}\right\Vert
_{2,Q}^{2}
+ \sum_{l=1}^{L}\left\Vert
\phi_{l}-s_{l}\right\Vert
_{1,Q}^{2}
\leq Ce^{-bW}.\label{eqn8s1}
\end{align}
It is easy to show that
\begin{align}
\sum_{l=1}^{L}\left\Vert
\psy_{l}-\ub_{l}\right\Vert
_{2,Q}^{2}
+\sum_{l=1}^{L}\left\Vert\phi_{l}-p_{l}\right\Vert
_{1,Q}^{2}
\leq  Ce^{-bW}.\label{eqn9s1}
\end{align}
Combining~\eqref{eqn8s1} and~\eqref{eqn9s1} we obtain the result.
\begin{rem}
 We can define a set of corrections to the obtained non-conforming solution  as defined in \cite{KI2}, \cite{SM3}
  such that both velocity and pressure variable result in a conforming set of solutions $(\zb,q)$
and we obtain the error estimate
\begin{align}
 \|\ub-\zb\|_{1,\Omega}+\|p-q\|_{0,\Omega}\leq ce^{-bW}.
\end{align}
\end{rem}
\section{Numerical results}
In this section, we present numerical results on Stokes equations with different non-standard boundary conditions.
We denote the error between $\ub$ and its approximate solution $\zb$ in $\Hb^{1}$ norm by $\|E_{\mathbf{u}}\|_{1}$, the error between $p$ and  $q$ in $L^2$ norm by $\|E_{p}\|_{0}$. $\|E_{c}\|_{0}$ denotes the error in continuity equation in $L^2$ norm and it is used to measure mass conserving property of the scheme. $W$ denotes the polynomial order and 'itr' represents the number of iterations
required for convergence. Here all the computations have been carried out in single element except in examples 5, 6 and 7. Multi element refinement can be done 
easily for each case. 

Preconditioned conjugate gradient method is used to solve the normal equations. Quadratic form defined in
(\ref{precon}) is used as a preconditioner.
The structure of the preconditioner consists of 3 blocks for each element, 
   the first two blocks correspond to $H^{2}$ norm of spectral element function defined for velocity variable and third block correspond to 
   $H^{1}$ norm of spectral element function defined for pressure variable. We conclude that the condition number of the preconditioned system is $O{(ln W)}^2$, hence with increase in polynomial order condition number increases.

 We verify the exponential accuracy of the proposed schemes. However order of accuracy and the number of iterations may 
 vary from one example to another. For example, boundary conditions with derivative terms need more iterations compared to the boundary conditions without having any derivative terms. More iterations are required to approximate non-algebraic solutions compared to algebraic solutions.
\subsection{Ex-1 : Stokes equations with boundary condition (B14)}
Consider Stokes equations on $\Omega=[0,1]^{2}$. 
Let $\Gamma_{0}=\left\{(0,x_{2}):0\leq x_{2}\leq 1\right\}\cup\left\{(x_{1},1):0\leq x_{1}\leq 1\right\}
                \cup\left\{(1,x_{2}):0\leq x_{2}\leq 1\right\}$ and $\Gamma_{1}=\left\{(x_{1},0):0\leq x_{1}\leq 1\right\}$.
We have chosen the data such that
\begin{align}
\begin{cases}
 &u_{1}=x_{1}^{2}(1-x_{1})^{2}(2x_{2}-6x_{2}^{2}+4x_{2}^{3}),\\
 &u_{2}=x_{2}^{2}(1-x_{2})^{2}(-2x_{1}+6x_{1}^{2}-4x_{1}^{3}),\\
 &p=x_{1}^{2}-x_{2}^{2},
 \end{cases}\notag
\end{align}
form a set of the exact solutions \cite{JI1}
with boundary conditions $\ub=0\quad\!\!\!\mbox{on}\quad\!\!\!\Gamma_{0}$ and 
  $\ub_{\tau}=0,\sigma_{\nb}=\omega_{\nb}\quad\!\!\!\mbox{on}\quad\!\!\!\Gamma_{1}$ (B14).
  Table 2 shows the errors $\|E_{\mathbf{u}}\|_{1}$ , $\|E_{p}\|_{0}$
and $\|E_{c}\|_{0}$  and the number of iterations for various values of $W$. One can see that the errors decay very fast
and the decay of the error $\|E_{c}\|_{0}$ confirms the mass conserving property of the numerical scheme. 
 \begin{table}[H]
~~~~~~~~~~~~~~~~~~~~~~~~~~~%
\begin{center}
\begin{tabular}{l ccccc}
\hline 
$W$ & $\|{E_{\ub}}\|_1$ & ${\|E_{p}\|}_0$  & $\|E_{c}\|_{0}$ & itr\\
\hline 
2 & 3.4205E-02 & 2.1800E-02 & 1.8922E-02 & 23\\
3 & 1.0216E-02 & 3.0624E-02 & 1.4758E-02 & 49\\
4 & 4.6465E-04 & 9.6598E-04 & 1.4123E-04 & 63\\
5 & 7.1188E-05 & 1.8080E-04 & 2.3220E-05 & 95\\
6 & 7.7329E-06 & 2.5381E-05 & 4.2328E-06 & 133\\
7 & 8.2112e-07 & 1.4825E-06 & 3.5594e-07 & 168\\
8 & 3.3948E-08 & 9.6400E-08 & 1.7772E-08 & 210\\
9 & 6.5440E-09 & 6.1380E-09 & 1.5928E-09 & 241\\
10 & 1.9301E-10 & 1.6941E-10 & 5.5645E-11 & 283\\
\hline 
\end{tabular}
\end{center}
\caption{$\|E_{\ub}\|_{1}$, $\|E_{p}\|_{0}$ and $\|E_{c}\|_{0}$ for
various values of $W$ for Ex-1}
\end{table}
\subsection{Ex-2 : Stokes equations with boundary condition (B15)}
Let $\Omega=[0,1]^{2}$. 
Let $\Gamma_{0}=\left\{(0,x_{2}):0\leq x_{2}\leq 1\right\}\cup\left\{(x_{1},1):0\leq x_{1}\leq 1\right\}
                \cup\left\{(1,x_{2}):0\leq x_{2}\leq 1\right\}$ and $\Gamma_{1}=\left\{(x_{1},0):0\leq x_{1}\leq 1\right\}$ .
We have chosen the data with a set of the exact solutions \cite{JI1}
\begin{align}
\begin{cases}
 u_{1}=sin(\pi x_{1})sin(\pi x_{2}),\\
 u_{2}=sin(\pi x_{1})sin(\pi x_{2}),\\
 p=cos(\pi x_{1})exp(x_{1}x_{2}),
 \end{cases}\notag
\end{align}
with boundary conditions $\ub=0\quad\!\!\!\mbox{on}\quad\!\!\!\Gamma_{0}$ and 
  $\ub_{\nb}=0,\sigma_{\taub}=\omega_{\taub}\quad\!\!\!\mbox{on}\quad\!\!\!\Gamma_{1}$ (B15).
  Table 3 shows the errors $\|E_{\mathbf{u}}\|_{1}$ , $\|E_{p}\|_{0}$
and $\|E_{c}\|_{0}$  and the number of iterations for various values of $W$. The exact solution is non-algebraic in nature and we can see 
that the number of iterations is high compared to the number of iterations in example 1. 
  \begin{table}[H]
~~~~~~~~~~~~~~~~~~~~~~~~~~~%
\begin{center}
\begin{tabular}{l ccccc}
\hline 
$W$ & ${\|E_{\ub}\|}_1$ & $\|{E_p}\|_0$  & $\|E_{c}\|_{0}$ & itr\\
\hline 
2 & 7.0505E-01 & 1.5337E+00 & 4.0524E-01 & 15\\
3 & 1.0688E-01 & 1.0022E+00 & 6.2620E-02 & 91\\
4 & 6.8930E-03 & 2.4056E-02 & 2.202E+00 & 270\\
5 & 4.0976E-04 & 1.9562E-03 & 3.9979E-04 & 372\\
6 & 4.9890E-05 & 1.2332E-04 & 3.1869E-05 & 425\\
7 & 1.3691E-06 & 3.4167E-06 & 1.4384E-06 & 564\\
8 & 2.1645E-07 & 6.5360E-07 & 1.1044E-07 & 637\\
\hline 
\end{tabular}
\end{center}
\caption{$\|E_{\ub}\|_{1}$, $\|E_{p}\|_{0}$ and $\|E_{c}\|_{0}$ for
various values of $W$ for Ex-2}
\end{table}
Figure 1 presents the log of the errors $\|E_{\mathbf{u}}\|_{1}$ , $\|E_{p}\|_{0}$  against $W$.  The graph is almost linear and this confirms the exponetial 
accuracy of the numerical method. 
\begin{figure}
\begin{center}
~~~~~~~~~~~~\includegraphics[scale=0.65]{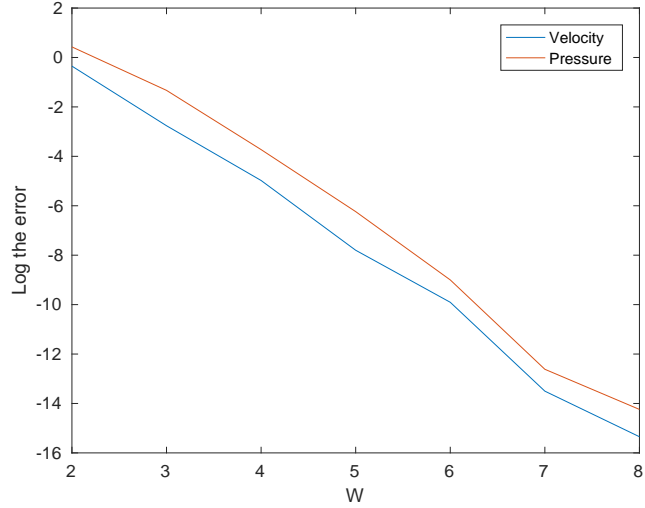}
\caption{Log of errors  $\|E_{\mathbf{u}}\|_{1}$ , $\|E_{p}\|_{0}$ against $W$ for Ex-2}
\end{center}
\end{figure}
\subsection{Ex-3: Stokes equations with boundary condition (B5)}
Consider Stokes equations with boundary conditions $p=0, \nb\times\ub=0$ (B5) with zero right hand side, on boundary of $\Omega=\left\{(-1,1)\times(-1,1)\right\}\setminus\left\{(0,1)\times(-1,0)\right\}$ such that
\begin{align}
 \begin{cases}
  u_{1}=-x_{2}(x_{2}^{2}-1),\\
  u_{2}=-x_{1}({x_{1}}^2-1),\\
  p=x_{1}x_{2}({x_{1}}^2-1)({x_{2}}^2-1),\notag
 \end{cases}
\end{align}
form a set of the exact solutions \cite{DU1}.
Table 4 shows the errors $\|E_{\mathbf{u}}\|_{1}$ , $\|E_{p}\|_{0}$
and $\|E_{c}\|_{0}$  and the number of iterations for various values of $W$. The numerical results confirm the
fast convergence of the numerical scheme.
\begin{table}[H]
~~~~~~~~~~~~~~~~~~~~~~~~~~~%
\begin{center}
\begin{tabular}{l ccccc}
\hline 
$W$ & $\|E_{\ub}\|_1$ & $\|E_{p}\|_0$  & $\|E_{c}\|_{0}$ & itr\\
\hline 
2 & 1.2903E-01 & 9.5655E-02 & 1.1322E-01 & 24\\
3 & 1.3147E-03 & 5.8071E-04 & 3.9079E-04 & 32\\
4 & 7.0454E-04 & 5.2770E-05 & 2.9696E-05 & 48\\
5 & 1.9117E-05 & 1.7295E-05 & 5.1751E-06 & 69\\
6 & 1.2540E-07 & 4.7325E-08 & 4.0931E-08 & 97\\
\hline 
\end{tabular}
\end{center}
\caption{$\|E_{\ub}\|_{1}$, $\|E_{p}\|_{0}$ and $\|E_{c}\|_{0}$ for
various values of $W$ for Ex-3}
\end{table}
\subsection{Ex-4: Stokes equations with boundary condition (B12)}
We consider Stokes equations on 
$\Omega=[-1,1]^{2}$, 
\begin{align}
 \begin{cases}
  \Gamma_{1}=\left\{(x_{1},x_{2}):-1<x_{1}<1, x_{2}=\pm 1\right\},\\ 
  \Gamma_{2}=\left\{(x_{1},x_{2}): x_{1}=-1, -1<x_{2}<1\right\},\\
  \Gamma_{3}=\left\{(x_{1},x_{2}): x_{1}=1, -1<x_{2}<1\right\}.
  \end{cases}\notag
\end{align}  
Force function and boundary data are chosen such that 
\begin{align}
 \begin{cases}
  u_1=sin\pi x_{1} sin\pi x_{2}\\
  u_2=cos\pi x_{1} cos\pi x_{2}\\
  p=x_{1}x_{2}
 \end{cases}\notag
\end{align}
form a set of the exact solutions to the  Stokes equations with boundary condition (B12).
The given domain is divided into four elements  with uniform step sizes one in each direction. The approximate solution is obtained and Table 5 shows the relative errors $\frac{\|E_{\mathbf{u}}\|_{1}}{\left\Vert \mathbf{u}\right\Vert _{1}}$ , $\frac{\|E_{p}\|_{0}}{\left\Vert p\right\Vert _{0}}$
and $\|E_{c}\|_{0}$  and the number of iterations for various values of $W$. The number of iterations is high because of the non-algebraic nature of the exact solution and also the pressure and mixed conditions on the boundary. 
\begin{table}[H]
~~~~~~~~~~~~~~~~~~~~~~~~~~~%
\begin{center}
\begin{tabular}{l ccccc}
\hline 
$W$ & $\frac{\|E_{\mathbf{u}}\|_{1}}{\left\Vert \mathbf{u}\right\Vert _{1}}$ & $\frac{\|E_{p}\|_{0}}{\left\Vert p\right\Vert _{0}}$  & $\|E_{c}\|_{0}$ & itr\\
\hline 
2 & 6.0646E-01 & 7.3160E+00 & 5.3414E+00 & 11\\
3 & 3.6217E-01 & 2.8203E+00 & 1.8667E+00 & 24\\
4 & 8.6433E-02 & 9.4839E-01 & 5.2985E-01 & 40\\
5 & 1.8758E-02 & 2.0515E-01 & 1.0918E-01 & 116\\
6 & 5.2405E-03 & 5.2657E-02 & 2.4415E-02 & 200\\
7 & 8.7689E-04 & 8.9130E-03 & 4.4200E-03 & 356\\
8 & 7.4496E-05 & 8.2511E-04 & 4.0602E-04 & 722\\
9 & 9.3018E-06 & 9.6641E-05 & 4.9851E-05 & 1153\\
10 & 2.0333E-06 & 2.1330E-05 & 1.1979E-05 & 1675\\
\hline 
\end{tabular}
\end{center}
\caption{$\frac{\|E_{\mathbf{u}}\|_{1}}{\left\Vert \mathbf{u}\right\Vert _{1}}$, $\frac{\|E_{p}\|_{0}}{\left\Vert p\right\Vert _{0}}$ and $\|E_{c}\|_{0}$ for various values of $W$ for Ex-4}
\end{table}
Figure 2 shows the log of the relative errors $\frac{\|E_{\mathbf{u}}\|_{1}}{\left\Vert \mathbf{u}\right\Vert _{1}}, \frac{\|E_{p}\|_{0}}{\left\Vert p\right\Vert _{0}} $ against $W$.  The graph is almost linear and this shows that method is exponentially 
accurate. 
\begin{figure}
\begin{center}
~~~~~~~~~~~~\includegraphics[scale=0.65]{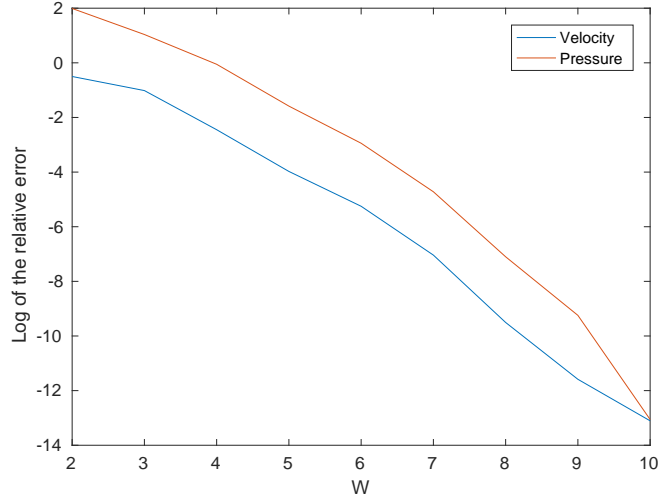}
\caption{Log of relative errors against $W$ for Ex-4}
\end{center}
\end{figure}
\subsection{Ex-5: Stokes equations with boundary conditions (B7) on $[-1,1]^{2}$}
Here we consider Stokes equations with the following boundary conditions
on $[-1,1]^{2}$ :
\begin{align}
\begin{cases}
\mathbf{u.n}=g,[(2\mathbb{D}\mathbf{u})\mathbf{n}]_{\taub}=\hb\,\,\textrm{on}\,\,\{(x_{1},x_{2}):-1<x_{1}<1,x_{2}=1\}\\
\mathbf{u=g_{1}}\,\textrm{\,on\,\,the\,\,other\,\,sides\,\,of}\,\,[-1,1]^{2}.
\end{cases}\notag
\end{align}
Data is chosen such that the exact solution of the boundary value
problem is 
\begin{align}
\begin{cases}
u_{1}=sin\,\pi x_{1}\,sin\,\pi x_{2},\\
u_{2}=cos\,\pi x_{1}\,cos\,\pi x_{2}\\
p=x_{1}x_{2\,\,\,\,\,\,\,\,\,\,\,\,\,\,\,\,\,\,\,}.
\end{cases}\notag
\end{align}
The given domain is divided into four square elements with equal sides of
length one. The error in the approximate solution is obtained for
different values of $W.$ Table 6 shows the relative errors
$\frac{\left\Vert E_{\mathbf{u}}\right\Vert _{1}}{\left\Vert \mathbf{u}\right\Vert _{1}},\frac{\left\Vert E_{p}\right\Vert _{0}}{\left\Vert p\right\Vert _{0}}$
and $\left\Vert E_{c}\right\Vert _{0}$ for different values of $W.$ The decay of the error $\|E_c\|_{0}$ shows the mass conserving 
property of the numerical scheme.
\begin{table}[H]
~~~~~~~~~~~~~~~~~~~~~~~~~~~%
\begin{center}
\begin{tabular}{l ccccc}
\hline 
$W$ & $\frac{\left\Vert E_{\mathbf{u}}\right\Vert _{1}}{\left\Vert \mathbf{u}\right\Vert _{1}}$ & $\frac{\left\Vert E_{p}\right\Vert _{0}}{\left\Vert p\right\Vert _{0}}$ & $\left\Vert E_{c}\right\Vert _{0}$ & itr\\
\hline 
2    & 7.5989E-01 & 2.9508E+00 & 5.7610E+00 & 19\\
4    & 6.1126E-02 & 4.2430E-01 & 3.6661E-01 & 176\\
6    & 1.7586E-03 & 1.2257E-02 & 1.1238E-02 & 421\\
8    & 7.2036E-05 & 5.7808E-04 & 3.6999E-04 & 734\\
10   & 1.9453E-06 & 9.8102E-05 & 9.6509E-06 & 1693\\
\hline 
\end{tabular}
\end{center}
\caption{$\frac{\left\Vert E_{\mathbf{u}}\right\Vert _{1}}{\left\Vert \mathbf{u}\right\Vert _{1}},\frac{\left\Vert E_{p}\right\Vert _{0}}{\left\Vert p\right\Vert _{0}}$
and $\left\Vert E_{c}\right\Vert _{0}$ for different values of $W$ for Ex-5}
\end{table}
\subsection{Ex-6: Stokes equations with boundary conditions (B3) on $[0,1]^{2}$}

Here we consider Stokes equations with the following boundary conditions
on $[0,1]^{2}$ 
\begin{eqnarray*}
\begin{cases}
\mathbf{u.n}=g,[(\nabla\mathbf{u}+(\nabla\mathbf{u})^{T}-pI)\mathbf{n}+\mathbf{u}]_{\mathbf{\taub}}=\mathbf{h}\,\,\textrm{on}\,\,\{(x_{1},x_{2}):0<x_{1}<1,x_{2}=0\}\\
\mathbf{u=g_{1}}\,\textrm{\,on\,\,the\,\,other\,\,sides\,\,of}\,\,[0,1]^{2}.
\end{cases}
\end{eqnarray*}

Data is chosen such that the exact solution of the boundary value
problem is 
\begin{eqnarray*}
\begin{cases}
u_{1}=sin\,\pi x_{1}\,sin\,\pi x_{2},\\
u_{2}=cos\,\pi x_{1}\,cos\,\pi x_{2},\\
p=x_{1}x_{2}.
\end{cases}
\end{eqnarray*}

The given domain is divided into four square elements with equal side
length $h=0.5.$ The error in the approximate solution is obtained
for different values of $W.$ Table 7 shows the relative
errors $\frac{\left\Vert E_{\mathbf{u}}\right\Vert _{1}}{\left\Vert \mathbf{u}\right\Vert _{1}},\frac{\left\Vert E_{p}\right\Vert _{0}}{\left\Vert p\right\Vert _{0}}$
and $\left\Vert E_{c}\right\Vert _{0}$ for different values of $W.$
\begin{table}[H]
~~~~~~~~~~~~~~~~~~~~~~~~~~~%
\begin{center}
\begin{tabular}{l ccccc}
\hline 
$W$ & $\frac{\left\Vert E_{\mathbf{u}}\right\Vert _{1}}{\left\Vert \mathbf{u}\right\Vert _{1}}$ & $\frac{\left\Vert E_{p}\right\Vert _{0}}{\left\Vert p\right\Vert _{0}}$ & $\left\Vert E_{c}\right\Vert _{0}$ & itr\tabularnewline
\hline 
2 & 1.6561E-01 & 2.4997E+00 & 4.7065E-01 & 10\tabularnewline

4 & 1.0226E-02 & 1.1310E-01 & 2.6468E-02 & 107\tabularnewline
 
6 & 7.2611E-04 & 8.7936E-03 & 1.9605E-03 & 419\tabularnewline
 
8 & 5.3046E-05 & 6.2276E-04 & 1.6031E-04 & 1091\tabularnewline
\hline 
\end{tabular}
\end{center}
\caption{$\frac{\left\Vert E_{\mathbf{u}}\right\Vert _{1}}{\left\Vert \mathbf{u}\right\Vert _{1}},\frac{\left\Vert E_{p}\right\Vert _{0}}{\left\Vert p\right\Vert _{0}}$
and $\left\Vert E_{c}\right\Vert _{0}$ for different values of $W.$}

\end{table}

\subsection{Ex-7: Stokes equations with boundary conditions (B10) on $[0,1]^{2}$}

Here we consider Stokes equations with the following boundary conditions
on $[0,1]^{2}$ 
\begin{eqnarray*}
\begin{cases}
\mathbf{u_{\taub}}=0,[(\nabla\mathbf{u}-pI)\mathbf{n}].\textrm{\ensuremath{\mathbf{n}}}=\mathbf{g}\,\,\textrm{on}\,\,\{(x_{1},x_{2}):0<x_{1}<1,x_{2}=0\}\\
\mathbf{u=g_{1}}\,\textrm{\,on\,\,the\,\,other\,\,sides\,\,of}\,\,[0,1]^{2}.
\end{cases}
\end{eqnarray*}

Data is chosen such that the exact solution of the boundary value
problem is 
\begin{eqnarray*}
\begin{cases}
u_{1}=sin\,\pi x_{1}\,sin\,\pi x_{2},\\
u_{2}=cos\,\pi x_{1}\,cos\,\pi x_{2},\\
p=x_{1}x_{2}.
\end{cases}
\end{eqnarray*}

The given domain is divided into four square elements with equal side
lengths $h=0.5.$ The error in the approximate solution is obtained
for different values of $W.$ Table 8 shows the relative
errors $\frac{\left\Vert E_{\mathbf{u}}\right\Vert _{1}}{\left\Vert \mathbf{u}\right\Vert _{1}},\frac{\left\Vert E_{p}\right\Vert _{0}}{\left\Vert p\right\Vert _{0}}$
and $\left\Vert E_{c}\right\Vert _{0}$ for different values of $W.$ The decay of the errors confirm the exponential decay
of the numerical method. The decay of the error confirms the exponential accuracy of the numerical method.
\begin{table}[H]
~~~~~~~~~~~~~~~~~~~~~~~~~~~%
\begin{center}
\begin{tabular}{l ccccc}
\hline 
$W$ & $\frac{\left\Vert E_{\mathbf{u}}\right\Vert _{1}}{\left\Vert \mathbf{u}\right\Vert _{1}}$ & $\frac{\left\Vert E_{p}\right\Vert _{0}}{\left\Vert p\right\Vert _{0}}$ & $\left\Vert E_{c}\right\Vert _{0}$ & itr\tabularnewline
\hline 
2 & 1.7125E-01 & 1.9638E+00 & 5.2266E+00 & 9\tabularnewline

4 & 1.1853E-02 & 1.6104E-01 & 2.8415E-02 & 106\tabularnewline
 
6 & 5.7420E-04 & 7.0487E-03 & 1.6417E-03 & 402\tabularnewline
 
8 & 4.4148E-05 & 5.2272E-04 & 1.2899E04 & 1039\tabularnewline
\hline 
\end{tabular}
\end{center}
\caption{$\frac{\left\Vert E_{\mathbf{u}}\right\Vert _{1}}{\left\Vert \mathbf{u}\right\Vert _{1}},\frac{\left\Vert E_{p}\right\Vert _{0}}{\left\Vert p\right\Vert _{0}}$
and $\left\Vert E_{c}\right\Vert _{0}$ for different values of $W.$}
\end{table}
\subsection{Ex-8: Stokes equations with boundary condition (B5) on $[-1,1]^{3}$}
Here we consider Stokes equations with boundary conditions $\ub\times\nb=\gb\times\nb, p=\phi$ on $\Omega= [-1,1]^{3}$. The data are chosen such that 
\begin{align*}
&u_{1}(x_{1},x_{2},x_{3})  =4x_{1}^{2}x_{2}x_{3}(1-x_{1})^{2}(1-x_{2})(1-x_{3})(x_{3}-x_{2}),\\
&u_{2}(x_{1},x_{2},x_{3})  =4x_{1}x_{2}^{2}x_{2}x_{3}(1-x_{1})(1-x_{2})^{2}(1-x_{3})(x_{1}-x_{3}),\\
&u_{3}(x_{1},x_{2},x_{3})  =4x_{1}x_{2}x_{3}^{2}(1-x_{1})(1-x_{2})(1-x_{3})^{2}(x_{2}-x_{1}),\\
&p(x_{1},x_{2},x_{3})  =-2x_{1}x_{2}x_{3}+x_{1}^{2}+x_{2}^{2}+x_{3}^{2}+x_{1}x_{2}+x_{1}x_{3}+x_{2}x_{3}\\
&\quad\quad\quad\quad\quad\quad\quad-x_{1}-x_{2}-x_{3},
\end{align*}
form an exact solution to Stokes problem with boundary condition (B5).
Here we consider a single element i.e $[-1,1]^3$. Table 9 shows the errors $\|E_{\mathbf{u}}\|_{1}$ , $\|E_{p}\|_{0}$
and $\|E_{c}\|_{0}$  and the number of iterations for various values of $W$.
\begin{table}[H]
~~~~~~~~~~~~~~~~~~~~~~~~~~~%
\begin{center}
\begin{tabular}{l ccccc}
\hline 
$W$ & ${\|E_{\ub}\|}_{1}$ & ${\|E_{p}\|}_{0}$  & ${\|E_{c}\|}_{0}$ & itr\\
\hline 
2 & 1.1089E+01 & 3.6122E-01 & 2.5517E+00 & 21\\
4 & 8.2164E-03 & 5.7646E-03 & 3.6079E-03 & 91\\
6 & 2.3867E-04 & 1.1298E-04 & 6.6529E-05 & 432\\
8 & 4.5072E-05 & 2.1736E-05 & 1.1137E-05 & 1512\\
\hline 
\end{tabular}
\end{center}
\caption{$\|E_{\ub}\|_{1}$, $\|E_{p}\|_{0}$ and $\|E_{c}\|_{0}$ for
various values of $W$ for Ex-8}
\end{table}
\section{Conclusion and future work}
Here we have proposed a unified approach for Stokes equations with various non-standard boundary
conditions based on a non-conforming least-squares spectral element method. A minimal amount of change, only related to the boundary terms is required in the 
implementation part with the change of boundary conditions. Exponential accuracy has been observed with various boundary conditions in different test cases. For some cases, we have plotted the logarithm of error versus degree of the polynomial, 
which comes out to be almost straight line, which graphically ensures exponential accuracy. For some of the test problems, iteration
count is high which is due to the complexity in approximating non-algebraic solutions and also boundary conditions involving derivatives. We have presented numerical results for seven different types of boundary conditions. One can see similar results in other cases. 

We expect similar behaviour for Navier-Stokes equations with these non-standard boundary conditions
to yield similar order accuracy, which is an ongoing investigation.
\section*{Declaration}
Authors declare that they do not have any conflict of interest.


\begin{thebibliography}{00}
\bibitem{AB1} H. Abboud, F.E. Chami and T. Sayah, A priori
and a posteriori estimates for three dimensional Stokes
equations with non-standard boundary conditions, Numer. Methods Partial Differ. Equ., 
28(4), 1178-1193, 2012.
\bibitem{AB2} H. Abboud, F.E. Chami and T. Sayah, Error estimates 
for three-dimensional Stokes problem with non-standard boundary
conditions, C.R. Acad. Sci. Paris, 349, 523-528, 2011.
\bibitem{AC1} P. Acevedo, C. Amrouche, C. Conca and A. Ghosh, Stokes 
and Navier-Stokes equations with Navier boundary condition, Comptes Rendus Mathematique,
357(2), 115-119, 2019.
\bibitem{ADN} S. Agmon, A. Douglis and L. Nirenberg, Estimates near
the boundary for solutions of elliptic partial differential
equations satisfying general boundary conditions II, Comm. Pure
Appl. Math. 17, 35-92, 1964.
\bibitem{AM4} M. Amara, E.C. Vera and D. Trujillo, Vorticity-velocity
-pressure formulations for Stokes problem, Math. Comput.,
73(248), 1673-1697, 2003.
\bibitem{AM7} K. Amoura, C. Bernardi, N. Chorfi and S. Saadi, Spectral discretization 
of the Stokes problem with mixed boundary conditions, Progress in Computational Physics, 20, 42-61, 2012.
\bibitem{AM6} C. Amrouche, C. Bernardi, M. Dauge and V. Girault, Vector
potentials in three dimensional nonsmooth domains, Math. Methods Appl. Sci., 21, 823-864, 1998.
\bibitem{AM5} C. Amrouche and A. Rejaiba, $L^{p}$ theory for Stokes and
 Navier Stokes equations with Navier boundary condition, J. Diff. Eq., 
 256, 1515-1547, 2014.
\bibitem{AM2} C. Amrouche and N.E.H. Seloula, Stokes equations and
elliptic systems with non standard boundary conditions, Comptes Rendus Math., 
 349 (11-12), 704-708, 2011.
\bibitem{AM3} C. Amrouche and N.H. Seloula, On the Stokes
equations with Navier-type boundary conditions, Diff. Eq.
Appl., 3(4), 581-607, 2011.
\bibitem{AM}C. Amrouche and N.H. Seloula, $L^{p}$ theory for 
vector potentials and Sobolev inequalities for vector fields : 
Application to the Stokes equations with pressure boundary
conditions, Math. Methods Appl. Sci., 23, 37-92, 2013.
\bibitem{BA1} G.R. Barrenechea, M. Basy and V. Dolean, Stabilized
hybrid discontinuous Galerkin methods for the Stokes problem with
non-standard boundary conditions, Lecture notes in computational
engineering, 179-189, 2018.
\bibitem{BE8} C. Begue, C. Conca, F. Murat and O. Pironneau, Les équations de
Stokes et de Navier-Stokes avec des conditions aux limites sur la pression,
Nonlinear Partial Differential Equations and Their Applications, College de
France Seminar, Pitman Res. Notes in Math., 181, 179-264, 1989.
\bibitem{BE9} A. Bendali, J.-M. Dominguez and S. Gallic, 
 A Variational Approach for
the Vector Potential Formulation of the Stokes and Navier-Stokes Problems in
Three Dimensional Domains, J. Math. Anal, and Appl., 107, 537-560, 1985.
\bibitem{BE1} J.M. Bernard, Non-standard Stokes and Navier-Stokes
problems : existence and regularity in stationary case,
Math. Meth. Appl. Sci., 25, 627-661, 2002.
\bibitem{BE4} J.M. Bernard, Spectral discretizations of the Stokes equations
with non-standard boundary conditions,J Sci Comput., 20(3), 355-377, 2004.
\bibitem{BE2} C. Bernardi, C. Canuto and Y. Maday , Spectral approximations of the 
Stokes equations with boundary conditions on the pressure, SIAM. J. Numer. Anal.,
28(2), 333-369, 1991.
\bibitem{BE3} C. Bernardi and N. Chorfi, Spectral discretization of the vorticity, velocity
and pressure formulation of the Stokes problem, SIAM J. Numer. Anal.,44(2), 826-850, 2006.
\bibitem{BE6} C. Bernardi, F. Hecht and F.Z. Nouri, A new finite element discretization of the Stokes problem coupled with Darcy equations, IMA J. Numer. Anal., 30, 61-93, 2010.
\bibitem{BE7} C. Bernardi, F. Hecht and O. Pironneau, Coupling Darcy and Stokes equations for porous media with cracks, Math. Model. Nemer. Anal., 39, 7-35, 2005.
\bibitem{BE5} S. Bertoluzza, V. Chabannes, C. Prud'homme and M. Szopas, Boundary conditions
involving pressure for the Stokes problem and applications in computational hemodynamics,
Comput. Methods Appl. Mech. Eng., 322, 58-80, 2017.
\bibitem {BD} P.B. Bochev and M.D. Gunzburger, Finite element methods of 
least-squares type, SIAM Rev., 40(4), 789-837, 1998.
\bibitem{BD1}  P.B. Bochev and M.D. Gunzburger, Least-squares methods for 
the velocity-pressure-stress formulation of the Stokes equations,
Comput. Methods Appl. Mech. Engrg , 126, 267-287, 1995.
\bibitem{BR1} J.H. Bramble and P. Lee, On variational formulations
for the Stokes equations with nonstandard boundary conditions, Math. Model. Numer. Anal., 
28(7), 903-919, 1994.
\bibitem{CA1}Z. Cai, T.A. Manteuffel and S.F. Mccormick, First
-order system least squares for velocity-vorticity-pressure form
of the Stokes equations, with application to linear elasticity, Electron. Trans. Numer. Anal.,
3, 150-159, 1995.
\bibitem {CH1} C.L. Chang and S.Y. Yang, Analysis of the $L^{2}$ least-squares 
finite element method for the velocity-vorticity-pressure Stokes equations 
with velocity boundary conditions, Appl. Math. Comput., 130, 121-144, 2002.
\bibitem {CH2} C.L. Chang and J. Nelson, Least-squares finite element method 
for the Stokes problem with zero residual of mass conservation,
SIAM J. Numer. Anal., 34, 480-489, 1997.
\bibitem{CO1} C. Conca, C. Pares, O. Pironneau and M. Thrift, Navier-Stokes
equations with imposed pressure and velocity fluxes, Int. J. Numer. Methods Fluids, 
 20, 267-287, 1995.
\bibitem{CO2} C. Conca, F. Murat and O. Pironneau, The Stokes and Navier-Stokes
equations with boundary conditions involving the pressure, Japan J. Math.,
20(2), 279-318, 1994.
\bibitem  {DG} J.M. Deang and M.D. Gunzburger, Issues related to least-squares 
finite element methods for the Stokes equations, SIAM J. Sci. Comput., 20, 878-906, 1999.
\bibitem{DO1} J.-M. Dominguez, Formulations en Potential Vecteur du système de
Stokes dans un Domaine de $R^3$ , Pub. lab. An. Num., L.A., 189, 1983.
\bibitem{DU1} S. Du and H. Duan, Analysis of a stabilized finite element method for 
Stokes equations of velocity boundary condition and pressure boundary condition, 
J. Comput. Appl. Math.,  337(C), 290-318, 2018.
\bibitem {KI2}P. K. Dutt, N. Kishore Kumar and C. S. Upadhyay, Nonconforming h-p spectral element methods for elliptic problems, 
Proc. Indian Acad. Sci (Math. Sci.), 117, 109-145, 2007.
%
\bibitem{GI1} V. Girault, Incompressible Finite Element Methods for Navier-Stokes
équations with Nonstandard Boundary Conditions in $R^3$ , Math. Comp., 51, 55-
74, 1988.
\bibitem{GO1} K. P. Gostaf and O. Pironneau, 
Pressure boundary conditions for blood flows, Chin. Ann. Math., Ser.
B, 36, 829-842, 2015.
\bibitem{HU1} T.R. Hughes, L.P. Franca, A new finite element formulation for computational fluid
dynamics : VII. The Stokes problem  with various well-posed boundary conditions : symmetric formulations that converge 
for all velocity/pressure spaces, Comput. methods Appl. Mech., 65, 85-96, 1987.
\bibitem{JI1} B.-N. Jiang and C.L. Chang, Least-squares 
finite elements for Stokes problem, Comput. Methods Appl. Mech, 78, 297-311, 1990.
\bibitem{KIM1} S.D. Kim, H.C. Lee and B.C. Shin, Least-squares
spectral collocation method for the Stokes equations, Numer. Methods Partial Differ. Equ.,
 20(1), 128-139, 2003.
\bibitem {JN1}B.N. Jiang, On the least-squares method, Comput. Methods Appl. Mech. Engrg,
          152, 239-257, 1998.
\bibitem{ME1} D. Medkova, Several non-standard problems for the stationary Stokes problem,
Analysis, 40(1), 1-17, 2020.
\bibitem{ME2} D. Medkova, One problem of the Navier type for the Stokes system in planar domains, J. Diff. Equ., 261, 5670-5689, 2016.
\bibitem{SM3} S. Mohapatra, Pravir K. Dutt, B. V. Rathish Kumar and
Marc I. Gerritsma, Non-conforming least squares spectral element method
for Stokes equations on non-smooth domains, J. of Comp. Appl. Math.,
372, 112696, 2020. 
\bibitem{SM1} S. Mohapatra and A. Husain, Least-squares spectral element method for three dimensional
             Stokes equations, Applied Numer. Math., 102, 31-54, 2016.
\bibitem{SM2} S. Mohapatra and S. Ganesan, Non-Conforming least squares spectral element formulation
for Oseen Equations with applications to Navier-Stokes equations, Num. Fun. Ana. and Opti.,
37(10), 1295-1311, 2016.
\bibitem{OL1} M.A. Ol'shanskii, On the Stokes problem with model boundary conditions, Mat. Sb.,
               188(4), 603-620, 1997.
\bibitem{OH1}  Y. Ohhara, M. Oshima, T. Iwai, H. Kitajima, Y. Yajima, K. Mitsudo, A. Krdy, and I. Tohnai, 
Investigation of blood flow in the external carotid artery and its branches with a new 0D peripheral model,
BioMed. Eng. OnLine,  15:16, 2016.             
\bibitem{PR1} M.M.J Proot and M.I. Gerritsma, A least-squares spectral element formulation for 
                the Stokes problem, J. Sci. Comput., 17, 285-296, 2002.
\bibitem {PR2} M. Proot and M.I. Gerritsma, Least-squares spectral elements applied to the
         Stokes problem, J. Comput. Phys., 181, 454-477, 2002.
\bibitem{PR3} M.M.J Proot and M.I. Gerritsma, Mass and momentum conservation of the least-squares 
        spectral element method for Stokes problem, J. Sci. Comput., 27, 389-401, 2006.
\bibitem{RU1} A. Russo and A. Tartaglione, On the Navier problem for the stationary Navier-Stokes
        equations, J. Diff. Eq., 251, 2387-2408, 2011.
\bibitem {CH} Ch. Schwab, $p$ and $h-p$ Finite element methods, Clarendon Press, Oxford, 1998.
\bibitem{SA1} N. Saito, On the Stokes equations with the leak and slip boundary conditions
of friction type : regularity of solutions, Publ. RIMS, Kyoto Univ., 40, 345-383, 2004.
\bibitem{SA2} S. Saito and L.E. Scriven, Study of coating flow by the finite element method, J. Comp. Phys.,
42, 53-76, 1981.
\bibitem{SI1} J. Silliman and L.E. Scriven, Separating flow near a static contact line :  slip at a wallsand shape of
a free surface,J. Comp. Phys.,  34, 287-313, 1980.
\bibitem{TA1} P.A. Tapia, C. Amrouche, C. Conca and A. Ghosh, Stokes and Navier-Stokes equations
with Navier boundary conditions, J Diff. Equ. , 285, 258-320, 2021.
\bibitem{TO} S.K. Tomar, h-p Spectral element methods for elliptic problems on non-smooth domains
using parallel computers, Ph.D. thesis (India: IIT Kanpur) (2001); Reprint available as
Tec. Rep. no. 1631, Department of Applied Mathematics, University of Twente, The
Netherlands. http://www.math.utwente.nl/publications.
\bibitem{TO2} Y. Tokuda, M.-H. Song, Y. Ueda, A. Usui, T. Akita, S. Yoneyama and S. Maruyama,
Three dimensional numerical simulation of blood flow in the aortic arch during cardiopulmnary bypass, 
Eur. J. Cardiothorac. Surg., 33, 164-167, 2008.
\bibitem{VE2} R. Verfurth, Finite element approximation on incompressible Navier-Stokes equations with
slip boundary condition, Numerische Mathematik, 50, 697-721, 1986.
\bibitem{VE3} R. Verfurth, Mixed Finite Element Approximation of the Vector
Potential, Numer. Math., 50, 685-695, 1987.
\bibitem{VE1} H.B.D. Viega, Regularity for Stokes and generalized Stokes systems under nonhomogeneous
slip-type boundary conditions, Adv. Diff. Equ., 9 (9-10), 1079-1114, 2004.
\bibitem{VI1} I. E. Vignon-Clementel, C. A. Figueroa, K. E. Jansen and C. A. Taylor, Outflow boundary
conditions for three dimensional finite element modeling of blood flow and pressure arteries, Comput. Methods
Appl. Mech. Engrg., 195, 3776-3796, 2006.
\bibitem{VI2}I. E. Vignon-Clementel, C. A. Figueroa, K. E. Jansen and C. A. Taylor,
Outflow boundary conditions for 3D simulations of non-periodic blood flow and pressure fields in
deformable arteries. Comput. Methods Biomech. Biomed. Engin., 13, 625-640, 2010.
\end{thebibliography}
\end{document}